\documentclass{amsart}
\usepackage{amsmath}
\usepackage{amssymb}
\usepackage{amsthm}
\usepackage{enumerate}
\usepackage[pdftex]{graphicx}
\usepackage{caption}
\usepackage{amscd}
\theoremstyle{definition}
\newtheorem{definition}{Definition}[section]
\theoremstyle{plain}
\newtheorem{lemma}[definition]{Lemma}
\newtheorem{theorem}[definition]{Theorem}

\newtheorem{proposition}[definition]{Proposition}
\newtheorem{corollary}[definition]{Corollary}
\theoremstyle{remark}
\newtheorem{remark}[definition]{Remark}

\makeatletter
\@namedef{subjclassname@2020}{
  \textup{2020} Mathematics Subject Classification}
\makeatother
\scrollmode
\newcommand{\mycl}{\operatorname{cl}}
\newcommand{\myint}{\operatorname{int}}

\newcommand{\myrank}{\operatorname{rank}}

\newcommand{\myIso}{\operatorname{iso}}
\newcommand{\myLpt}{\operatorname{lpt}}
\newcommand{\myedim}{\operatorname{eRank}}
\newcommand{\mypdeg}{\operatorname{p.deg}}
\newcommand{\myReg}{\operatorname{Reg}}

\begin{document}
\title[Definable $\mathcal C^r$ structures on definable topological groups]{Definable $\mathcal C^r$ structures on definable topological groups in d-minimal structures}
\author[M. Fujita]{Masato Fujita}
\address{Department of Liberal Arts,
Japan Coast Guard Academy,
5-1 Wakaba-cho, Kure, Hiroshima 737-8512, Japan}
\email{fujita.masato.p34@kyoto-u.jp}

\begin{abstract}
Definable topological groups whose topologies are affine have definable $\mathcal C^r$ structures in d-minimal expansions of ordered fields, where $r$ is a positive integer.
We prove this fact using a new notion called partition degree of a definable set.
Basic properties of partition degree are also studied.
\end{abstract}



\maketitle

\section{Introduction}\label{sec:intro}
Pillay proved that any group $G$ definable in an o-minimal structure can be equipped with a definable topology which makes $G$ an abstract definable manifold whose group operation and inverse induced from $G$ are continuous \cite{Pillay2}. 
We can easily extend it to the definable $\mathcal C^r$ category when the o-minimal structure is an expansion of an ordered field. 
The purpose of this paper is to prove an assertion weaker than but similar to Pillay's results in d-minimal expansions of ordered fields; that is, for any positive integer $r$, any definable topological group $G$ is definably homeomorphic to an abstract definable $\mathcal C^r$ manifold whose group operation and inverse induced from $G$ are of class $\mathcal C^r$.
See Definition \ref{def:mfd}, Definition \ref{def:crstr} and Theorem \ref{thm:cr_structure_on_group} for a precise description of the main theorem.

We introduce relevant works.
Mosley observed that Pillay's method works for sufficiently saturated first-order topological structures on which natural conditions are imposed such as the exchange property of algebraic closure \cite{Morley}.
First-order topological structures are defined in another Pillay's paper \cite{Pillay1}.
Every expansion of a linear order is a first-order topological structure, and the structures discussed in this paper are first-order topological structures. 

In \cite{Wencel}, Wencel gave a purely topological proof for Pillay's results.
His method is applicable to all first-order topological structures which have dimension functions satisfying van den Dries's requirements \cite[Definition]{vdD2} and the continuity property.
The continuity property means that, for any definable function $f:X \to F$, where $F$ is the universe, the set of points at which $f$ is discontinuous is of dimension smaller than $\dim X$. 
An o-minimal structure satisfies the above conditions \cite[Chapter 4]{vdD}.
More generally, a definably complete locally o-minimal structure also satisfies the above conditions \cite[Proposition 2.8]{FKK}.
Recall that o-minimality implies definable completeness and local o-minimality. 

D-minimality is a weaker concept than definably complete local o-minimality; that is, a definably complete locally o-minimal structure is always d-minimal. 
Miller first defined the notion of d-minimality when the universe is the set of reals \cite{Miller-dmin} and Fornasiero generalized Miller's definition.
Several structures other than definably complete locally o-minimal structures are known to be d-minimal \cite{FM,MT}.
We give a definition of d-minimality below.
We employed Fornasiero's definition in \cite{Fornasiero}.
\begin{definition}
	An expansion of a dense linear order without endpoints $\mathcal F=(F,<,\ldots)$ 
	is \textit{definably complete} if every definable subset of $F$ has both a supremum and an infimum in 
	$F \cup \{ \pm \infty\}$ \cite{M}.
	The structure $\mathcal F$  is \textit{d-minimal} if it is definably complete, and every definable subset $X$ of $F$ is the union of an open set and finitely many discrete sets, where the number of discrete sets does not depend on the parameters of definition of $X$	\cite{Fornasiero}.
\end{definition}
D-minimal structures enjoy some common tame topological properties possessed by definably complete locally o-minimal structures.
See \cite{Fuji_tame, FKK} for studies of tameness of sets definable in definably complete locally o-minimal structures.
For instance, the standard topological dimension function defined in Definition \ref{def:dim} satisfies van den Dries's requirements in d-minimal expansions of ordered fields as it does in the definably complete locally o-minimal case.
However, the standard topological dimension function does not necessarily have the continuity property in d-minimal structures.
It implies that Wencel's results \cite{Wencel} are not applicable to d-minimal structures.
We need new tools to prove our main theorem.

Partition degree and hugeness are key tools developed for our proof of the main theorem.
The author believes that they are also useful for other studies of tameness of sets definable in d-minimal structures. 
For a given definable set $X$ of dimension $d$, there exists a definable $\mathcal C^r$ submanifold of dimension $d$ which is contained and open in $X$.
The largest definable $\mathcal C^r$ submanifold of dimension $d$ satisfying the above condition is denoted by $\myReg_r(X)$.
The difference $X \setminus \myReg_r(X)$ may be still of dimension $d$, but it contains a definable $\mathcal C^r$ submanifold of dimension $d$ which is open in it.
We obtain a definable set of dimension smaller than $d$ after repeatedly removing definable $\mathcal C^r$ submanifolds of dimension $d$ from $X$ finitely many times.
The partition degree of $X$ is the number of removed definable $\mathcal C^r$ submanifolds through the above process. 
Hugeness is an alternative tool to $d$-largeness used in Wencel's study \cite{Wencel}.
A definable subset $X$ of a definable set $Y$ is huge if and only if $\dim (\myint_Y(Y \setminus X))<\dim Y$.
Partition degree and hugeness are defined and their properties are investigated in Section \ref{sec:partition_degree}.

Section \ref{sec:defnable_grpup} is devoted to the introduction of definitions used in the main theorem such as abstract definable $\mathcal C^r$ manifolds and a proof of the main theorem. 
Section \ref{sec:preliminary} is a preliminary section and collects basic assertions proved in the previous works.

We introduce the terms and notations used in this paper. 
Throughout, the term ‘definable’ means ‘definable in the given structure with parameters.’
For any expansion of a linear order whose universe is $F$, we assume that $F$ is equipped with the order topology induced from the linear order $<$ and the topology on $F^n$ is the product topology of the order topology on $F$ unless the topology in consideration is explicitly described.
An open box in $F^n$ is the product of $n$ many nonempty open intervals.
We set $|x|=\max\{|x_i|\;|\; 1 \leq i \leq n\}$ for $x=(x_1,\ldots, x_n) \in F^n$.
For a topological space $X$, $\myint_{X}(A)$, $\mycl_{X}(A)$ and $\partial_{X}(A)$ denote the interior, the closure and the frontier of a subset $A$ of $X$ in $X$, respectively.
We drop the subscript $X$ when the topology space $X$ is obvious from the context.

\section{Preliminary}\label{sec:preliminary}
We recall several basic assertions used in this paper.
\begin{theorem}[Definable inverse function theorem]\label{thm:inverse}
Let $r$ be a positive integer.
Consider a definably complete expansion of an ordered field $\mathcal F=(F,<,+,\cdot,0,1,\ldots)$.	
Let $f:U \to F^m$ be a definable $\mathcal C^r$ map defined on a definable open subset $U$ of $F^m$.
Take a point $a \in U$ so that the differential $d_af:F^m \to F^m$ of $f$ at the point $a$ is invertible.
Then there are definable open neighborhoods $U'$ of the point $a$ contained in $U$ and $V'$ of $f(a)$ such that the restriction $f|_{U'}:U' \to V'$ of $f$ to $U'$ is a definable $\mathcal C^r$ diffeomorphism onto $V'$.   
\end{theorem}
\begin{proof}
	A proof is found in \cite[Chapter 7, Theorem 2.11]{vdD} when the structure is o-minimal.
	Extreme value theorem is only used in the proof, therefore; this theorem holds for every expansion of an ordered field in which extreme value theorem holds.
	Extreme value theorem holds for every definably complete structures by \cite[Corollary(Max-min theorem)]{M}.
\end{proof}

We recall the definition of dimension and its basic properties.
\begin{definition}[Dimension]\label{def:dim}
	Consider an expansion of a dense linear order without endpoints.
	Let $F$ be the universe.
	We consider that $F^0$ is a singleton with the trivial topology.
	Let $X$ be a nonempty definable subset of $F^n$.
	The dimension of $X$ is the maximal nonnegative integer $d$ such that $\pi(X)$ has a nonempty interior for some coordinate projection $\pi:F^n \rightarrow F^d$.
	We set $\dim(X)=-\infty$ when $X$ is an empty set.
\end{definition}

\begin{proposition}\label{prop:dim}
	Consider a d-minimal expansion of an ordered field $\mathcal F=(F,<,+,\cdot,0,1,\ldots)$.
	The dimension function $\dim$ fulfills the conditions in Definition of \cite{vdD2}; that is,
	\begin{enumerate}
		\item[(1)] $\dim(S)= -\infty \Leftrightarrow S=\emptyset$; $\dim(\{x\})=0$ for all $x \in F$ and $\dim F=1$.
		\item[(2)] $\dim (S_1 \cup S_2) = \max\{\dim S_1, \dim S_2\}$ for any definable subsets of $F^n$.
		In particular, if the union $S_1 \cup S_2$ has a nonempty interior, at least one of $S_1$ and $S_2$ has a nonempty interior.
		\item[(3)] $\dim S^\sigma=\dim S$ for any definable set $S \subseteq F^n$ and any permutation $\sigma$ of $\{1,\ldots,n\}$.
		Here, $S^\sigma=\{(x_{\sigma(1)},\ldots, x_{\sigma(n)}) \in F^n\;|\; (x_1,\ldots, x_n) \in S\}$.
		\item[(4)] Let $T$ be a definable subset of $F^{n+1}$ and set $T_x=\{y \in F\;|\;(x,y) \in T\}$ for any $x \in F^n$.
		Set $T(i)=\{x \in F^n\;|\; \dim(T_x)=i\}$ for $i=0,1$.
		Then, $T(i)$ are definable and $\dim(\{(x,y) \in T\;|\; x \in T(i)\})=\dim T(i) +i$ for $i=0,1$.
	\end{enumerate}
\end{proposition}
\begin{proof}
	See \cite[Lemma 4.5]{Fornasiero} together with \cite{Miller-choice} other than the `in particular' part of assertion (2).
	The `in particular' part is obvious from the formula and the definition of dimension.
\end{proof}

\begin{corollary}\label{cor:dim1}
	Let $\mathcal F=(F,<,+,\cdot,0,1,\ldots)$ be as in Proposition \ref{prop:dim}.
	The following assertions hold:
\begin{enumerate}
	\item[(1)] Let $f:X \rightarrow F^n$ be a definable map. 
	We have $\dim(f(X)) \leq \dim X$.
	\item[(2)] Let $\varphi:X \rightarrow Y$ be a definable surjective map whose fibers are equi-dimensional; that is, the dimensions of the fibers $\varphi^{-1}(y)$ are constant.
	We have $\dim X = \dim Y + \dim \varphi^{-1}(y)$ for all $y \in Y$.  
\end{enumerate}
\end{corollary}
\begin{proof}
	The assertions (1) and (2) hold by \cite{vdD2} when the dimension function fulfills conditions in Definition of \cite{vdD2}.
	The corollary follows from Proposition \ref{prop:dim}.  
\end{proof}

We collect basic but important assertions on d-minimal structures from \cite{Fornasiero}.

\begin{lemma}\label{lem:dim1}
	Let $\mathcal F=(F,<,+,\cdot,0,1,\ldots)$ be as in Proposition \ref{prop:dim}.
	Let $X$ be a definable subset of $F^n$ of dimension $d$ and let $\pi:F^n \to F^d$ be a coordinate projection such that $\pi(X)$ has a nonempty interior.
	Then there exists a definable dense open subset $U$ of $F^d$ such that 
	$\pi^{-1}(x) \cap X$ is of dimension $\leq 0$ and the equality $\mycl_{F^n}(X) \cap \pi^{-1}(x) = \mycl_{F^n}(X \cap \pi^{-1}(x)) $ holds for each $x \in U$. 
\end{lemma}
\begin{proof}
	See \cite[Theorem 3.10(8)]{Fornasiero}.
\end{proof}

\begin{lemma}\label{lem:dim2}
	Let $\mathcal F=(F,<,+,\cdot,0,1,\ldots)$ be as in Proposition \ref{prop:dim}.
	Let $U$ be a nonempty definable open subset of $F^n$ and $f:U \to F$ be a definable function.
	There exists a nonempty definable open subset $V$ of $U$ such that the restriction of $f$ to $V$ is continuous.
\end{lemma}
\begin{proof}
	See \cite[Theorem 3.10(6)]{Fornasiero}.
\end{proof}

\begin{lemma}\label{lem:dim3}
	Let $r$ be a positive integer.
	Let $\mathcal F=(F,<,+,\cdot,0,1,\ldots)$ be as in Proposition \ref{prop:dim}.
	Let $U$ be a nonempty definable open subset of $F^n$ and $f:U \to F$ be a definable function.
	There exists a definable closed subset $D$ of $U$ having an empty interior such that $f$ is of class $\mathcal C^r$ on $U \setminus D$.
\end{lemma}
\begin{proof}
	See \cite[Lemma 3.14]{Fornasiero}.
\end{proof}

\begin{lemma}\label{lem:dim4}
	Let $\mathcal F=(F,<,+,\cdot,0,1,\ldots)$ be as in Proposition \ref{prop:dim}.
	We have $\dim \mycl(X) = \dim X$ for each definable set $X$.
\end{lemma}
\begin{proof}
	See \cite[Theorem 3.10(5)]{Fornasiero}.
\end{proof}

\begin{lemma}\label{lem:dmin_char}
	Let $\mathcal F=(F,<,+,\cdot,0,1,\ldots)$ be as in Proposition \ref{prop:dim}.
	Let $\pi:F^{m+n} \to F^m$ be a coordinate projection and $X$ be a definable subset of $F^{m+n}$.
	There exists a positive integer $N$ such that, for each $x \in F^m$, either $\dim (X \cap \pi^{-1}(x))>0$ or $X \cap \pi^{-1}(x)$ is a union of at most $N$ many definable discrete sets.
\end{lemma}
\begin{proof}
	See \cite[Lemma 5.10]{Fornasiero}.
\end{proof}

The following definable choice lemma is also used:
\begin{lemma}\label{lem:definable_selection}
	Consider a d-minimal expansion of an ordered group whose universe is $F$.
	Let $\pi:F^n \to F^m$ be a coordinate projection and $X$ be a definable subset of $F^n$.
	Then there exists a definable map $\varphi:\pi(X) \to X$ such that the composition $\pi \circ \varphi$ is the identity map on $\pi(X)$.
\end{lemma}
\begin{proof}
	See \cite{Miller-choice}.
\end{proof}

An assertion similar to the following proposition is found in \cite[Theorem 3.11]{Fuji_tame} when the structure is definably complete and locally o-minimal.
The following proposition is proven similarly to \cite[Theorem 3.11]{Fuji_tame}, but we give a complete proof here for readers' convenience.

\begin{proposition}\label{prop:dim2}
	Let $\mathcal F=(F,<,+,\cdot,0,1,\ldots)$ be as in Proposition \ref{prop:dim}.
	A definable set $X$ is of dimension $d$ if and only if the nonnegative integer $d$ is the maximum of nonnegative integers $e$ such that there exist an open box $B$ in $F^e$ and 
	a definable injective continuous map $\varphi:B \rightarrow X$ homeomorphic onto its image. 
\end{proposition}
\begin{proof}
	Set $d=\dim X$ and let $d'$ be the maximum of nonnegative integers $e$ having an open box $B$ in $F^e$ and 
	a definable injective continuous map $\varphi:B \rightarrow X$ satisfying the condition in the proposition.
	We first show the inequality $d' \leq d$.
	Let $B$ and $\varphi:B \rightarrow X$ be an open box and a definable injective continuous map satisfying the above condition, respectively.
	We have $\dim \varphi(B) = \dim B =d'$ by Corollary \ref{cor:dim1}(1).
	We get $d=\dim X \geq \dim \varphi(B)=d'$.
	
	We next prove the inequality $d \leq d'$.
	Let $F^n$ be the ambient space of $X$.
	By the definition of dimension, there exists a coordinate projection $\pi:F^n \to F^d$ such that $\pi(X)$ has a nonempty interior.
	Let $U$ be an open box contained in $\pi(X)$.
	By Lemma \ref{lem:definable_selection}, there exists a definable map $\tau:U \to X$ such that the composition $\pi \circ \tau$ is the identity map.
	We may assume that $\tau$ is continuous by Lemma \ref{lem:dim2} by shrinking the definable open set $U$ if necessary.
	Since $\tau$ is a definable continuous injective map homeomorphic onto its image, we have $d \leq d'$.
\end{proof}

The following corollary is also proven when the structure is definably complete and locally o-minimal in \cite[Corollary 3.12]{Fuji_tame}.
We give a complete proof here for the same reason as above.
\begin{corollary}\label{cor:dim2}
Let $\mathcal F=(F,<,+,\cdot,0,1,\ldots)$ be as in Proposition \ref{prop:dim}.
Let $X$ be a nonempty definable subset of $F^n$ of dimension $d$.
There exists a point $x \in X$ such that, for any open box $B$ containing the point $x$ such that $\dim X \cap B = d$.
\end{corollary}
\begin{proof}
	By Proposition \ref{prop:dim2}, there exists an open box $U$ in $F^d$ and definable injective continuous map $\varphi:U \to X$ homeomorphic onto its image.
	Take $t \in U$ and set $x=\varphi(t)$.
	For any open box $B$ containing the point $x$, the inverse image $\varphi^{-1}(B)$ is a definable open subset of $F^d$.
	We can take an open box $V$ containing the point $t$ and contained in $\varphi^{-1}(B)$.
	The restriction $\varphi|_{V}: V \to X \cap B$ of $\varphi$ to $V$ is a definable injective continuous map homeomorphic onto its image.
	It implies that $\dim (X \cap B) \geq d$ by Proposition \ref{prop:dim2}.
	The opposite inequality $\dim (X \cap B) \leq \dim (X) = d$ is obvious.
\end{proof}

\section{Partition Degree and Hugeness}\label{sec:partition_degree}
We consider a d-minimal expansion of an ordered field.
The dimension function introduced in Section \ref{sec:preliminary} possesses tame features, but it is not sufficiently tame for our study.
It does not necessarily fulfill the continuity property.
We introduce a new function called partition degree which supplements the dimension function.
First of all, we introduce the notion of regularity, which is used in the definition of partition degree.

\subsection{Regularity and Definable $\mathcal C^r$ submanifolds}\label{subsec:regularity}

We  define $r$-regularity and definable $\mathcal C^r$ submanifolds.
\begin{definition}
	Let $r$ be a positive integer.
	Consider a definably complete expansion of an ordered field, whose universe is denoted by $F$.
	A definable subset $X$ of $F^n$ of dimension $d$ is called a \textit{definable $\mathcal C^r$ submanifold} of dimension $d$ if, for every point $x$ in $X$, there exists a definable $\mathcal C^r$ diffeomorphism $\varphi:U \to V$ from a definable open neighborhood $U$ of $x$ in $F^n$ onto a definable open neighborhood $V$ of the origin in $F^n$ such that $\varphi(x)$ is the origin and the equality $\varphi(X \cap U)= V \cap (F^d \times \{\overline{0}_{n-d}\}) $ holds.
	Here, $\overline{0}_{n-d}$ denotes the origin in $F^{n-d}$.
	A definable map $f:X \to F$ defined on a definable $\mathcal C^r$ submanifold $X$ is \textit{of class $\mathcal C^r$} if the composition $f \circ (\varphi|_{X \cap U})^{-1}$ is of class $\mathcal C^r$, where $\varphi|_{X \cap U}$ denotes the restriction of the definable $\mathcal C^r$ diffeomorphism $\varphi:U \to V$ to $X \cap U$.
	
	Let $X$ be a definable subset of $F^n$ of dimension $d$.
	We say that a point $x$ in $F^n$ is \textit{$r$-regular} in $X$ if $x \in X$ and there exists a definable $\mathcal C^r$ diffeomorphism $\varphi:U \to V$ from 
	 a definable open neighborhood $U$ of $x$ in $X$ which is simultaneously a definable $\mathcal C^r$ submanifold onto a definable open subset $V$ of $F^d$.
	We simply call an $r$-regular point \textit{regular} when the positive integer $r$ is clear from the context.
\end{definition}	

The following lemmas are trivial:
\begin{lemma}\label{lem:open_mfd}
	Let $r$ be a positive integer.
	Consider a definably complete expansion of an ordered field.
	A definable open subset of a definable $\mathcal C^r$ submanifold of dimension $d$ is a definable $\mathcal C^r$ submanifold of dimension $d$.
\end{lemma}

\begin{lemma}\label{lem:open_regular}
	Let $r$ be a positive integer.
	Consider a definably complete expansion of an ordered field.
	Let $x$ be a point in a definable set $X$.
	Let $U$ be a definable open subset of $X$ containing the point $x$.
	Assume that $\dim X=\dim U$.
	The point $x$ is $r$-regular in $X$ if and only if it is $r$-regular in $U$.
\end{lemma}

The following lemma says that the standard topological dimension of a definable $\mathcal C^r$ submanifold of dimension $d$ equals to $d$. 
\begin{lemma}\label{lem:dim_mfd}
Let $r$ be a positive integer.
Consider a d-minimal expansion of an ordered field.
The equality $\dim M=d$ holds for every definable $\mathcal C^r$ submanifold $M$ of dimension $d$.
\end{lemma}
\begin{proof}
	Let $F$ be the universe and $F^n$ be the ambient space of $M$.
	Set $d'=\dim M$.
	We can take $x \in M$ such that $\dim M \cap B=d'$ for every open box $B$ containing the point $x$ by Corollary \ref{cor:dim2}.
	By the definition of definable $\mathcal C^r$ submanifold of dimension $d$, there exists a definable $\mathcal C^r$ diffeomorphism $\varphi:U \to V$ from a definable open neighborhood $U$ of $x$ in $F^n$ onto a definable open neighborhood $V$ of the origin in $F^n$ such that $\varphi(x)$ is the origin and the equality $\varphi(M \cap U)= V \cap (F^d \times \{\overline{0}_{n-d}\}) $ holds.
	We have $\dim M \cap U=d$ by Corollary \ref{cor:dim1}(1).
	We take an open box containing the point $B$ so that $B$ is contained in $U$.
	The inclusions $M \cap B \subseteq M \cap U \subseteq M$ imply $d' \leq d \leq d'$.
	We have proven $d=d'$.
\end{proof}

We give an equivalent condition for a point to be regular.
\begin{lemma}\label{lem:regular}
	Let $r$ be a positive integer.
	Consider a definably complete expansion of an ordered field.
	Let $F$ be its universe and $X$ be a definable subset of $F^n$ of dimension $d$.
	A point $x \in X$ is $r$-regular in $X$ if and only if there exist a coordinate projection $\pi:F^n \to F^d$ and an open box $U$ in $F^n$ containing the point $x$ such that $X \cap U$ is the permuted graph of a definable $\mathcal C^r$ map $\varphi$ defined on $\pi(U)$ under $\pi$; namely, $X \cap U$ is the graph of the map $\varphi$ after the permutation of coordinates so that $\pi$ is the projection onto the first $d$ coordinates.  
\end{lemma}
\begin{proof}
	The `if' part of the lemma is obvious.
	We prove the `only if' part.
	Assume that $x$ is $r$-regular in $X$.
	There exists a definable $\mathcal C^r$ diffeomorphism $\varphi:U' \cap X \to V$. Here, $U'$ is a definable open neighborhood of $x$ in $F^n$ and $V$ is a definable open subset of $F^d$.
	We may assume that $\varphi(x)$ is the origin $\overline{0}_d$ without loss of generality.
	Let $\psi:V \to U' \cap X$ be the inverse map of $\varphi$.
	Consider the differential $d_{\overline{0}_d}\psi: T_{\overline{0}_d}V \to F^n$ of $\psi$ at the origin $\overline{0}_d$.
	We have $\myrank d_{\overline{0}_d}\psi=d$ because $\psi$ is a diffeomorphism.
	There exists a coordinate projection $\pi:F^n \to F^d$ such that the composition $d_{\overline{0}_n} \pi \circ d_{\overline{0}_d}\psi$ is invertible.
	By Theorem \ref{thm:inverse}, there exists a definable open subset $W$ of $F^d$ and a definable $\mathcal C^r$ map $\tau:W \to V$ such that $\pi(x) \in V$, the composition $\pi \circ \psi$ is a definable $\mathcal C^r$ diffeomorphism on $\tau(W)$ and $\tau$ is its inverse.
	
	Set $f=\psi \circ \tau$.
	The map $f$ is a definable $\mathcal C^r$ map.
	We have $f(W) \subseteq X$ and $\pi \circ f$ is the identity map on $W$.
	Set $U''=U' \cap \pi^{-1}(W)$.
	The equality $U'' \cap X = f(W)$ is obvious.
	Take a sufficiently small $\varepsilon>0$ so that $B_n(x,\varepsilon):=\{y \in F^n \;|\; |x-y|<\varepsilon\}$ is contained in $U''$.
	Since $f$ is continuous, if we take a sufficiently small $\delta>0$, we have $f(B_d(\pi(x),\delta) )\subseteq B_n(x,\varepsilon)$.
	Here, $B_d(\pi(x),\delta) $ is defined in the same manner as $B_n(x,\varepsilon)$.
	Set $U=\pi^{-1}(B_d(\pi(x),\delta) ) \cap B_n(x,\varepsilon)$.
	The intersection $X \cap U$ is the permuted graph of a definable $\mathcal C^r$ map defined on $\pi(U)$ under $\pi$.
\end{proof}

We introduce several useful corollaries of Lemma \ref{lem:regular}.

\begin{definition}
	Let $r$ be a positive integer.
	Consider a definably complete expansion of an ordered field.
	Let $X$ be a definable set and $\myReg_r(X)$ denotes the set of $r$-regular points in $X$.
	We omit the subscript $r$ of $\myReg_r(X)$  when $r$ is clear from the context. 
\end{definition}

\begin{corollary}\label{cor:regular1}
	Let $r$ be a positive integer.
	Consider a d-minimal expansion of an ordered field.
	The set $\myReg_r(X)$ is definable and open in $X$.
	The equality $\dim \myReg_r(X) = \dim X$ holds when the structure is d-minimal.
\end{corollary}
\begin{proof}
	It follows from Lemma \ref{lem:regular} and Corollary \ref{cor:dim1}.
\end{proof}

\begin{corollary}\label{cor:regular2}
	Let $r$ be a positive integer.
	Consider a definably complete expansion of an ordered field whose universe is $F$.
	A definable subset $X$ of $F^n$ of dimension $d$ is a definable $\mathcal C^r$ submanifold of dimension $d$ if and only if every point in $X$ is $r$-regular in $X$.
\end{corollary}
\begin{proof}
	The `only if' part immediately follows from the definitions of $r$-regularity and definable $\mathcal C^r$ submanifolds.
	The `if' part follows from Lemma \ref{lem:regular}.
\end{proof}

\begin{corollary}\label{cor:product_reg2}
	Let $r$ be a positive integer.
	Consider a definably complete expansion of an ordered field.
	Let $X$ and $Y$ be definable sets.
	Let $Z$ be a definable subset of $X \times Y$ of dimension $=\dim X+\dim Y$.
	For any point $(a,b) \in \myReg_r(Z)$,  the fiber $Z_a:=\{y \in Y\;|\; (a,y) \in Z\}$ is of dimension $=\dim Y$ and  the point $b$ is $r$-regular in $Z_a$. 
\end{corollary}
\begin{proof}
	Let $F$ be the universe.
	Let $F^m$ and $F^n$ be the ambient spaces of $X$ and $Y$, respectively.
	Set $d=\dim X$ and $e=\dim Y$.
	By Lemma \ref{lem:regular}, there exist a coordinate projection $\pi:F^{m+n} \to F^{d+e}$ and an open box $B$ in $F^{m+n}$ containing the point $(a,b)$ such that $Z \cap B$ is the permutated graph of a definable $\mathcal C^r$ map $\varphi:\pi(B) \to F^{m+n-d-e}$ under $\pi$.
	In particular, the equality $\pi(Z \cap B)=\pi(B)$ holds.
	
	Take open boxes $B_1 \subseteq F^m$ and $B_2 \subseteq F^n$ so that $B=B_1 \times B_2$.
	There exist coordinate projections $\pi_1:F^m \to F^{d'}$ and $\pi_2:F^n \to F^{e'}$ such that $\pi(x,y) = (\pi_1(x),\pi_2(y))$ for each $x \in F^m$ and $y \in F^n$. 
	The equalities $\pi(B)=\pi_1(B_1) \times \pi_2(B_2)$ and $d'+e'=d+e$ are obvious by the definition of $\pi_1$ and $\pi_2$.
	We can obviously take definable $\mathcal C^r$ maps $\varphi_1:\pi(B) \to F^{m-d'}$ and $\varphi_2:\pi(B) \to F^{n-e'}$ such that  $\varphi(x,y)=(\varphi_1(x,y),\varphi_2(x,y))$ for each $x \in \pi_1(B_1) $ and $y \in \pi_2(B_2)$ by the definitions of $\pi$, $\pi_1$, $\pi_2$ and $\varphi$.

	We want to show the equalities $d=d'$ and $e=e'$.
	Let $\rho_1:F^{m+n} \to F^m$ and $\rho_2:F^{m+n} \to F^n$ be coordinate projections onto the first $m$ coordinates and onto the last $n$ coordinates, respectively.
	Note that $\rho_1(Z) \subseteq X$ and $\rho_2(Z) \subseteq Y$ because $Z$ is a subset of $X \times Y$.
	Set $\pi'_i=\pi_i \circ \rho_i$ for $i=1,2$.  
	The equality $\pi(Z \cap B)=\pi(B)$ implies the equality $\pi'_1(Z \cap B)=\pi'_1(B)$.
	This implies that the image $\pi'_1(Z)$ has a nonempty interior in $F^{d'}$.
	Therefore, $\pi_1(X)$ has a nonempty interior because $\rho_1(Z) \subseteq X$. 
	This means that $d' \leq d$.
	We have $e' \leq e$ in the same manner.
	We get $d'=d$ and $e'=e$ because $d'+e'=d+e$.
	
	The intersection $Z_a \cap B_2$ is the permutated graph of the definable $\mathcal C^r$ map defined by $\pi_2(B_2)  \ni x \mapsto \varphi_2(\pi_1(a),x) \in F^{n-e}$.
	This implies two facts.
	Firstly, the definable set $Z_a \cap B_2$ is of dimension not smaller than $e$. 
	Since $Z_a \cap B_2 \subseteq Z_a \subseteq Y$ and $\dim Y=e$, we get $\dim Z_a=e$.
	Secondly, it implies that $b$ is $r$-regular in $Z_a$ by Lemma \ref{lem:regular}. 
\end{proof}

\begin{corollary}\label{cor:product_reg}
	Let $r$ be a positive integer.
	Consider a definably complete expansion of an ordered field.
	Let $X$ and $Y$ be definable sets.
	The equality $\myReg_r(X \times Y)=\myReg_r(X) \times \myReg_r(Y)$ holds.
\end{corollary}
\begin{proof}
	The inclusion $\myReg_r(X) \times \myReg_r(Y) \subseteq \myReg_r(X \times Y)$ is easy to prove.
	We concentrate on the proof of the opposite inclusion.
	Take an arbitrary point $(a,b) \in \myReg_r(X \times Y)$ with $a \in X$ and $b \in Y$.
	The point $b$ is $r$-regular in $(X \times Y)_a=Y$ by Corollary \ref{cor:product_reg2}.
	It means that  $b \in \myReg_r(Y)$.
	We can prove $a \in \myReg_r(X)$ similarly.
\end{proof}

\begin{corollary}\label{cor:mfd_same_dim}
	Let $r$ be a positive integer.
	Consider a d-minimal expansion of an ordered field.
	Let $M$ be a definable $\mathcal C^r$ submanifold of dimension $d$.
	A definable subset $A$ of $M$ has a nonempty interior in $M$ if and only if $\dim A=d$.
	In addition, if $\dim A=d$, the equality $\dim (\myint_M(A))=d$ holds.
\end{corollary}
\begin{proof}
	The `only if' part follows from Lemma \ref{lem:open_mfd} and Lemma \ref{lem:dim_mfd}.
	
	We prove the `if' part.
	Let $F$ be the universe of the structure and $F^n$ be the ambient space of $M$.
	We can take a point $x \in A$ so that the equality $\dim A \cap B=d$ holds for any open box $B$ containing the point $x$ by Corollary \ref{cor:dim2}.
	By Lemma \ref{lem:regular} and Corollary \ref{cor:regular2}, if we take an open box $B$ appropriately, the intersection $B \cap M$ is the graph of a definable $\mathcal C^r$ map $\varphi:\pi(B) \to F^{n-d}$ after permuting the coordinates if necessary.
	Here, $\pi$ denotes the coordinate projection onto the first $d$ coordinates.
	By Corollary \ref{cor:dim1}(2), $\pi(A \cap B)$ has a nonempty interior because $\dim A \cap B=d$ and $A \cap B$ is the graph of a definable map defined on $\pi(A \cap B)$.
	Take a nonempty open box $V$ contained in $\pi(A \cap B)$.
	Set $U=B \cap \pi^{-1}(V)$.
	It is obvious that $U \cap M = U \cap A$ because $U \cap M \cap \pi^{-1}(x)$ is a singleton for each $x \in V$.
	It means that $A$ has a nonempty interior in $M$.
	
	The `in addition' part immediately follows from the equivalence we have just proven.
\end{proof}

\subsection{Definition of partition degree}\label{subsec:def_partition_degree}

We are now ready to define partition degree.
\begin{definition}
	Let $r$ be a positive integer.
	Consider a definably complete expansion of an ordered field.
	Let $F$ be its universe.
	The \textit{partition degree}, denoted by $\mypdeg(X)$, of a nonempty definable subset $X$ of $F^n$ of dimension $d$ is the minimum nonnegative integer $m$ such that $X$ is partitioned into $m+2$ definable subsets $X_{-1},X_0, \ldots, X_m$ satisfying the following conditions:
	\begin{enumerate}
		\item[(1)] $\dim X_{-1}<d$:
		\item[(2)] the definable set $X_i$ is a nonempty definable $\mathcal C^r$ submanifold of dimension $d$ and it is open in $\bigcup_{j=-1}^i X_j$ for each $0 \leq i \leq m$.
	\end{enumerate}
	Note that we do not require that $X_{-1}$ is nonempty.
	The sequence $X_{-1},X_0, \ldots, X_m$ of definable subsets of $X$ satisfying the above conditions is called an \textit{$r$-partition sequence}.
	The partition degree $\mypdeg(X)$ is independent of the choice of $r$ and finite when the structure is d-minimal.
	We prove these facts in this section.
\end{definition}

We first begin to show that $\mypdeg(X)$ is independent of the choice of $r$.
We need several lemmas.
\begin{lemma}\label{lem:mfd_cr}
	Let $r$ be a positive integer.
	Consider a d-minimal expansion of an ordered field $\mathcal F=(F,<,+,\cdot,0,1,\ldots)$.
	Let $M$ be a definable $\mathcal C^r$ submanifold of $F^n$ and $f:M \to F$ be a definable function.
	There exists a definable open subset $U$ in $M$ such that $\dim M \setminus U<\dim M$ and the restriction of $f$ to $U$ is of class $\mathcal C^r$.
\end{lemma}
\begin{proof}
	Set $d=\dim M$.
	Let $Z$ be the set of points at which $f$ is not of class $\mathcal C^r$.
	For any point $x \in M$, there exist a definable open neighborhood $U_x$ of $x$ in $M$ and a definable open subset $V_x$ of $F^d$ and a definable $\mathcal C^r$ diffeomorphism $\varphi_x:U_x \to V_x$.
	Apply Lemma \ref{lem:dim3} to the definable map $f \circ \varphi_x^{-1}:V_x \to F$.
	The set $\varphi_x(Z \cap U_x)$ has an empty interior.
	It means that $\dim \varphi_x(Z \cap U_x) <d$.
	We have $\dim Z \cap U_x < d$ by Corollary \ref{cor:dim1}(1).
	Since $x$ is an arbitrary point in $M$, we get $\dim Z<d$ by Corollary \ref{cor:dim2}.
	We have $\dim \mycl_{F^n}(Z)=\dim Z<d$ by Lemma \ref{lem:dim4}.
	Set $U=M \setminus \mycl_{F^n}(Z)$.
	Then $U$ is open, $f$ is of class $\mathcal C^r$ on $U$ and $\dim M \setminus U < d$.
\end{proof}

\begin{corollary}\label{cor:mfd_bdry}
Let $r$ be a positive integer.
Consider a d-minimal expansion of an ordered field $\mathcal F=(F,<,+,\cdot,0,1,\ldots)$.
Let $M$ be a definable $\mathcal C^r$ submanifold of $F^n$ and $U$ be a definable open subset of $M$.
Then the frontier $\partial_M(U)$ of $U$ in $M$ is of dimension smaller than $\dim M$. 
\end{corollary}
\begin{proof}
	Consider the definable function $f: M \to F$ which is zero on $U$ and one outside of $U$.
	The frontier $\partial_M(U)$ coincides with the set of points at which $f$ is not of class $\mathcal C^r$.
	The corollary follows from Lemma \ref{lem:mfd_cr}. 
\end{proof}

\begin{lemma}\label{lem:mfd_rr}
Let $r$ and $r'$ be positive integers.
Consider a d-minimal expansion of an ordered field $\mathcal F=(F,<,+,\cdot,0,1,\ldots)$.
Let $M$ be a definable $\mathcal C^r$ submanifold of $F^n$.
There exists a definable closed subset $Z$ of $M$ of dimension smaller than $\dim M$ such that $M \setminus Z$ is a definable $\mathcal C^{r'}$ submanifold of $F^n$.
\end{lemma}
\begin{proof}
	Let $Z'$ be the set of points in $M$ which is not $r'$-regular in $M$.
	We want to show that $\dim Z'<\dim M$.
	Once it is proved, we get the lemma.
	Indeed, set $Z=\mycl(Z') \cap M$. 
	We have $\dim Z = \dim Z' < \dim M$ by Lemma \ref{lem:dim4}.
	The difference $M \setminus Z$ is a definable $\mathcal C^{r'}$ submanifold of $F^n$ by Corollary \ref{cor:regular2}.
	
	Take an arbitrary point $x \in M$.
	We have only to prove that there exists an open box $U$ containing the point $x$ such that $\dim Z' \cap U < d:=\dim M$ by Corollary \ref{cor:dim2}.
	There exists an open box $U$ containing the point $x$ such that $M \cap U$ is the permuted graph of a definable $\mathcal C^r$ map $\varphi$ defined on $\pi(U)$ under a coordinate projection $\pi$ by  Lemma \ref{lem:regular} and Corollary \ref{cor:regular2}.
	We may assume that $\pi$ is the coordinate projection onto the first $d$ coordinates by permuting the coordinates if necessary.
	Let $f:\pi(U) \to X \cap U$ be the definable $\mathcal C^r$ map defined by $f(x)=(x,\varphi(x))$.
	There exists a definable closed subset $D$ of $U$ such that $\dim D < d$ and $f$ is of class $\mathcal C^{r'}$ on $U \setminus D$ by Lemma \ref{lem:dim3}.
	It is obvious that any point in $U \cap M \setminus f(D)$ is $r'$-regular in $M$ by Lemma \ref{lem:regular}.
	This implies the inclusion $Z' \cap U \subseteq f(D)$.
	On the other hand, we have $\dim f(D) \leq \dim D < d$ by Corollary \ref{cor:dim1}(1).
	We have proven that $\dim Z' \cap U <d$ as desired.
\end{proof}

\begin{proposition}\label{prop:independence}
Let $r$ and $r'$ be positive integers.
Consider a d-minimal expansion of an ordered field $\mathcal F=(F,<,+,\cdot,0,1,\ldots)$.
Let $X$ be a definable subset of $F^n$.
If $X$ has an $r$-partition sequence of length $l$, it also has an $r'$-partition sequence of length $l$.
In particular, the partition degree $\mypdeg(X)$ is independent of the choice of $r$. 
\end{proposition}
\begin{proof}
	The `in particular' part is obvious once the former part of the proposition is shown.
	We concentrate on the former part.
	
	If $r \geq r'$, the proposition is obvious because an $r$-partition sequence is always an $r'$-partition sequence.
	Consider the case in which $r<r'$.
	Set $d=\dim X$.
	Let $M_{-1}, M_0,\ldots, M_m$ be an $r$-partition sequence of $X$.
	Apply Lemma \ref{lem:mfd_rr} to $M_i$ for each $0 \leq i \leq m$, then there exists a definable closed subset $Z_i$ of dimension $<d$ and $N'_i:=M_i \setminus Z_i$ is a definable $\mathcal C^{r'}$ submanifold.
	Set $N_{-1}:=\mycl_X(M_{-1} \cup \bigcup_{i=0}^m Z_i)$, which is of dimension $<d$ by Lemma \ref{lem:dim4}.
	Set $N_i = N'_i \setminus N_{-1}$ for $0 \leq i \leq m$.
	The sequence  $N_{-1}, N_0,\ldots, N_m$ is an $r'$-partition sequence of $X$.
	We have proven the former part of the proposition.
\end{proof}

We have proven that the partition degree $\mypdeg(X)$ is independent of the choice of $r$. 
We use this fact without notice throughout this paper.

We next prove that $\mypdeg(X)$ is finite.
We need some preparation before we prove it.
We recall the definition of Cantor-Bendixson rank.
\begin{definition}[\cite{FM}]\label{def:lpt}
	We denote the set of isolated points in $S$ by $\myIso(S)$ for any topological space $S$.
	We set $\myLpt(S):=S \setminus \myIso(S)$.
	In other word, a point $x \in S$ belongs to $\myLpt(S)$ if and only if $x \in \mycl_S(S \setminus \{x\})$.
	
	Let $X$ be a nonempty closed subset of a topological space $S$.
	We set $X[ 0 ]=X$ and, for any $m>0$, we set $X [ m ] = \myLpt(X [ m-1 ])$.
	We say that $\myrank(X)=m$ if $X [ m ]=\emptyset$ and $X[ m-1 ] \neq \emptyset$.
	We say that $\myrank X = \infty$ when $X [ m ] \neq \emptyset$ for every natural number $m$.
	We set $\myrank(Y):=\myrank(\mycl_S(Y))$ when $Y$ is a nonempty subset of $S$ which is not necessarily closed.
\end{definition}

%

We next introduce the extended rank function $\myedim$ for d-minimal structures, which was first introduced in \cite{Fujita}.

\begin{definition}[Extended rank]\label{def:extended_dim_dmin}
	Consider a d-minimal expansion of an ordered field $\mathcal F=(F,<,+,\cdot,0,1,\ldots)$.
	Let $\Pi(n,d)$ be the set of coordinate projections of $F^n$ onto $F^d$.
	Recall that $F^0$ is a singleton.
	We consider that $\Pi(n,0)$ is a singleton whose element is a trivial map onto $F^0$.
	Since $\Pi(n,d)$ is a finite set, we can define a linear order on it. 
	We denote it by $<_{\Pi(n,d)}$.
	Let $\mathcal E_n$ be the set of triples $(d,\pi,r)$ such that $d$ is a nonnegative integer not larger than $n$, $\pi \in \Pi(n,d)$ and $r$ is a positive integer.
	The linear order $<_{\mathcal E_n}$ on $\mathcal E_n$ is the lexicographic order.
	We abbreviate the subscript $\mathcal E_n$ of $<_{\mathcal E_n}$ in the rest of the paper, but it will not confuse readers.
	Let $X$ be a nonempty bounded definable subset of $F^n$.
	The triple $(d,\pi,r)$ is the \textit{extended rank} of $X$ and denoted by  $\myedim_n(X)$ if it is an element of $\mathcal E_n$ satisfying the following conditions:
	\begin{itemize}
		\item $d = \dim X$;
		\item the projection $\pi$ is a largest element in $\Pi(n,d)$ such that $\pi(X)$ has a nonempty interior;
		\item  $r$ is a largest positive integer such that there exists a definable open subset $U$ of $F^d$ contained in $\pi(X)$ such that the set $\pi^{-1}(x) \cap X$ is of dimension zero and the equality $\myrank(\pi^{-1}(x) \cap X) = r$ holds  for each $x \in U$.
	\end{itemize}
	Note that such a positive integer $r$ exists by Lemma \ref{lem:dmin_char} and Proposition \ref{prop:dim}(2).
	We set $\myedim_n(\emptyset)=-\infty$ and define that $-\infty$ is smaller than any element in $\mathcal E_n$.
	
	Let us consider the case in which $X$ is an unbounded definable subset of $F^n$.
	Let $\varphi:F \to (-1,1)$ be the definable homeomorphism given by $\varphi(x)=\frac{x}{\sqrt{1+x^2}}$.
	We define $\varphi_n:F^n \to (-1,1)^n$ by $\varphi_n(x_1,\dots, x_n)=(\varphi(x_1),\ldots, \varphi(x_n))$.
	We set $\myedim_n(X)=\myedim_n(\varphi_n(X))$.
	
	Note that, when $X$ is a bounded definable subset of $F^n$, the equality $\myedim_n(X)=\myedim_n(\varphi_n(X))$ holds by \cite[Lemma 7.15]{Fujita}.
\end{definition}

\begin{lemma}\label{lem:formula_rank}
	Consider a d-minimal expansion of an ordered field $\mathcal F=(F,<,+,\cdot,0,1,\ldots)$.
	Let $A$ and $B$ be definable subsets of $F^n$.
	The equality $\myedim_n(A \cup B) = \max\{\myedim_n(A),\myedim_n(B)\}$ holds.
\end{lemma}
\begin{proof}
	See \cite[Proposition 7.19]{Fujita}.
\end{proof}

\begin{theorem}\label{thm:finite_pdeg}
Consider a d-minimal expansion of an ordered field $\mathcal F=(F,<,+,\cdot,0,1,\ldots)$.
The partition degree $\mypdeg(X)$ is finite for every nonempty definable set $X$.
\end{theorem}
\begin{proof}
	We fix a positive integer $p$.
	We prove the theorem by induction on $(d,\pi,r)=\myedim_n(X)$.
	Let $F^n$ be the ambient space of $X$.
	We first consider the case in which $(d,\pi,r)$ is the smallest element.
	We have $d=0$, $r=1$ and $\pi$ is the unique map from $F^n$ onto the singleton $F^0$.
	The definable set $X$ is discrete by \cite[Proposition 7.19]{Fujita} and it is closed by \cite[Lemma 7.16]{Fujita}.
	In particular, every point $X$ is $p$-regular in $X$.
	It means $X$ is a definable $\mathcal C^p$ submanifold of $F^n$ by Lemma \ref{cor:regular2}.
	We have $\mypdeg(X)=0$ in this case.
	
	We next consider the other case.
	We may assume that $\pi$ is the projection onto the first $d$ coordinates by permuting the coordinates if necessary.
	We reduce to simpler cases step by step.
	We use the following claim in the reductions:
	\medskip
	
	\textbf{Claim.} Let $X$ be a nonempty definable set.
	Set $(d,\pi,r)=\myedim_n(X)$.
	Let $Z$ be a definable subset of $F^d$ such that $\myedim_n(X \cap \pi^{-1}(Z))<\myedim_n(X)$.
	Set $U= F^d \setminus \mycl(Z)$.
	The theorem holds for $X$ if the theorem holds for $X \cap \pi^{-1}(U)$.  
	\begin{proof}[Proof of Claim]
		Note that $X \setminus \pi^{-1}(U)=X \cap \pi^{-1}(\mycl(Z))$.
		We show that $\myedim_n(X \setminus \pi^{-1}(U))<\myedim_n(X)$.
		Assume for contradiction that $\myedim_n(X \setminus \pi^{-1}(U))=\myedim_n(X)$.
		There exists a nonempty definable open subset $V$ of $F^d$ contained in $\mycl(Z)$ such that $X \cap \pi^{-1}(x)$ is of dimension zero and of rank $r$ for each $x \in V$.
		At least one of $\partial(Z \cap V)$ and $Z \cap V$ has a nonempty interior by Proposition \ref{prop:dim}(2).
		The set $\partial(Z \cap V)$ has an empty interior by the definition of the frontier.
		The definable set $Z \cap V$ has a nonempty interior, but it contradicts the assumption that $\myedim_n(X \cap \pi^{-1}(Z))<\myedim_n(X)$.
		We have proven $\myedim_n(X \setminus \pi^{-1}(U))<\myedim_n(X)$.
		
		We first consider the case in which $\dim (X \setminus \pi^{-1}(U))<\dim X$.
		There exists a $p$-partition sequence $N_{-1},N_0,\ldots, N_m$ of $X \cap \pi^{-1}(U)$ of finite length by the assumption.
		The sequence $N_{-1} \cup  (X \setminus \pi^{-1}(U)), N_0,\ldots, N_m$ is a $p$-partition chain of $X$ of finite length.
		
		We next consider the case in which $\dim (X \setminus \pi^{-1}(U))=\dim X$.
		There exists a $p$-partition sequence $N_{-1},N_0,\ldots, N_m$ of $X \cap \pi^{-1}(U)$ of finite length by the assumption.
		We can construct a $p$-partition sequence $M_{-1},M_0,\ldots, M_m$ of $X \setminus \pi^{-1}(U)$ of finite length by the induction hypothesis.
		We set $M'_i= M_i \setminus \mycl(N_{-1})$ and $N'_j=N_j \setminus \mycl(N_{-1})$ for each $0 \leq i \leq m$ and $0 \leq j \leq l$.
		The definable sets $M'_i$ and $N'_j$ are definable $\mathcal C^p$ submanifolds of $F^n$ by Lemma \ref{lem:open_mfd}.
		We have $\dim (M_{-1} \cup (\mycl(N_{-1}) \cap X)) < \dim X$ by Proposition \ref{prop:dim} and Lemma \ref{lem:dim4}.
		Since $X \cap \pi^{-1}(U)$ is open in $X$, the sequence $M_{-1} \cup (\mycl(N_{-1}) \cap X), M'_{0}, \ldots, M'_{l}, N'_{0}, \ldots, N'_{m}$ is a $p$-partition sequence of $X$.
	\end{proof}
	
	Set $U'_1= \{x \in \pi(X)\;|\; \dim(\pi^{-1}(x) \cap X)=0 \text{ and }\myrank(\pi^{-1}(x) \cap X)=r\}$.
	It is obvious that $\myedim_n(X \setminus \pi^{-1}(\myint(U'_1)))< \myedim_n(X)$ by the definition of extended rank.
	By Claim, we may assume that $\pi^{-1}(x) \cap X$ is of dimension $\leq 0$ and 	$\myrank(\pi^{-1}(x) \cap X)=r$ for each $x \in \pi(X)$ by replacing $X$ with $X \cap \pi^{-1}(\myint(U'_1))$.

	We set  $Y=\bigcup_{x \in \pi(X)} \myIso(\pi^{-1}(x) \cap X)$.
	We next reduce to the case in which $Y$ is open in $X$.
	We consider two separate cases.
	When $\pi(X \setminus Y)$ has an empty interior, it is obvious that $\myedim_n(X \setminus Y)<\myedim_n(X)$.
	Using Claim, we may assume that $Y=X$ by replacing $X$ with $X \setminus \pi^{-1}(\mycl(\pi(X \setminus Y)))$.  
	The definable set $Y$ is open in $X$ in this case.
	We consider the case in which $\pi(X \setminus Y)$ has a nonempty interior.
	Apply Lemma \ref{lem:dim1} to $X \setminus Y$.
	There exists a definable dense open subset $U_2$ of $F^d$ such that 
	$\pi^{-1}(x) \cap (X \setminus Y)$ is of dimension $\leq 0$ and $\mycl(X \setminus Y) \cap \pi^{-1}(x) = \mycl((X \setminus Y) \cap \pi^{-1}(x)) $ for each $x \in U_2$. 
	For any point $x \in F^d$, we have $\mycl(\pi^{-1}(x) \cap (X \setminus Y)) \cap Y=\emptyset$ because any point $Y \cap \pi^{-1}(x)$ is isolated in $X \cap \pi^{-1}(x)$ by the definition of $Y$.
	This implies that, for each $x \in U_2$, $\pi^{-1}(x) \cap \mycl(X \setminus Y) \cap Y=\emptyset$.
	This implies that $Y \cap \pi^{-1}(U_2)$ is open in $X \cap \pi^{-1}(U_2)$.
	Since $\pi(X \setminus \pi^{-1}(U_2))$ has an empty interior, we have $\myedim_n(X \setminus \pi^{-1}(U_2))< \myedim_n(X)$.
	We can reduce to the case in which $Y$ is open in $X$ by replacing $X$ with $X \cap \pi^{-1}(U_2)$ using Claim.
	
	Let $M$ be the set of points $y$ in $X$ such that there exists an open box $B$ in $F^n$ containing the point $y$ and the intersection $X \cap B$ is the graph of a definable $\mathcal C^p$ map defined on $\pi(B)$.
	For any $x \in F^d$, any point in $M \cap \pi^{-1}(x)$ is an isolated point in $X \cap \pi^{-1}(x)$.
	It means that $M$ is contained in $Y$.
	We show that $\pi(Y \setminus M)$ has an empty interior.
	Assume for contradiction that $\pi(Y \setminus M)$ contains a nonempty open box $D$.
	By Lemma \ref{lem:definable_selection}, there exists a definable map $\psi: D \to F^{n-d}$ whose graph is contained in $Y \setminus M$.
	Since $Y \cap \pi^{-1}(x)$ is discrete for each $x \in D$, we can construct a definable function $\rho: D \to F$ such that $\rho(x)>0$ and $(\{x\} \times B_{n-d}(\psi(x), \rho(x)) )\cap Y = \{(x,\psi(x))\}$ for each $x \in D$ by Lemma \ref{lem:definable_selection}.
	Here, $B_{n-d}(\psi(x), \rho(x)) =\{z \in F^{n-d}\;|\; |z-\psi(x)|<\rho(x)\}$.
	We may assume that $\psi$ and $\rho$ are of class $\mathcal C^p$ by Lemma \ref{lem:dim3} by shrinking $D$ if necessary.
	However, it contradicts the definition of $M$ because, for each $x \in D$, there exists an open box $B$ containing the point $(x,\psi(x))$ such that the intersection $B \cap X$ is the graph of the restriction $\psi|_{\pi(B)}$ of $\psi$ to $\pi(B)$.
	We have demonstrated that $\pi(Y \setminus M)$ has an empty interior.
	We can reduce to the case in which $Y=M$ by Claim by considering $X \setminus \pi^{-1}(\mycl(\pi(Y\setminus M)))$ in place of $X$.
	It is also obvious that $Y$ is open in $X$ after the replacement.
	
	The definable set $M$ is a definable $\mathcal C^p$ submanifold of $F^n$ and $Y=M$ is open in $X$.
	It is obvious that $\myrank((X \setminus Y) \cap \pi^{-1}(x))<r$ for every $x \in \pi(X)$ by the definition of $Y$.
	This implies the inequality $\myedim_n(X \setminus M)<(d,\pi,r)=\myedim_n(X)$ because $Y=M$ by the assumption.
	If $\dim X \setminus M<d=\dim X$, the sequence $X \setminus M, M$ of length two is a $p$-partition sequence of $X$.
	Otherwise, there exists a $p$-partition sequence of $X \setminus M$ of finite length, say, $N_{-1}, N_0, \ldots, N_m$ by the induction hypothesis.
	The sequence $N_{-1}, N_0, \ldots, N_m, M$ is a $p$-partition sequence of $X$.
	This implies that $\mypdeg(X)<\infty$.
\end{proof}

Theorem \ref{thm:finite_pdeg} guarantees that $\mypdeg(X)<\infty$ for every set $X$ definable in d-minimal expansions of ordered fields.
We use this fact without notice in this paper.
We can extend Theorem \ref{thm:finite_pdeg} to the parametrized case as follows:

\begin{theorem}\label{thm:finite_pdeg_parametrized}
	Consider a d-minimal expansion of an ordered field $\mathcal F=(F,<,+,\cdot,0,1,\ldots)$.
	Let $\pi:F^{m+n} \to F^m$ be the coordinate projection onto the first $m$ coordinate.
	Let $p$ be a positive integer.
	Let $X$ be a definable subset of $F^{m+n}$. 
	
	There exist finitely many definable subsets $M_{-1}, M_0, \ldots, M_q$ such that the sequence $M_{-1} \cap \pi^{-1}(x), \ldots, M_{q} \cap \pi^{-1}(x)$ is a $p$-regular sequence of $X \cap \pi^{-1}(x)$ for every $x \in \pi(X)$ after removing empty sets from the sequence other than $M_{-1} \cap \pi^{-1}(x)$.	
	
	In particular, the inequality $\mypdeg(X \cap \pi^{-1}(x)) \leq q$ holds for every $x \in \pi(X)$.
\end{theorem}

The proof of the above theorem is essentially the same as that of Theorem \ref{thm:finite_pdeg}, but it is more complicated. 
We first prove the following proposition:

\begin{proposition}\label{prop:uniform_bound_eRank}
	Consider a d-minimal expansion of an ordered field $\mathcal F=(F,<,+,\cdot,0,1,\ldots)$.
	Let $X$ be a nonempty definable subset of $F^{m+n}$ and $\pi:F^{m+n} \to F^m$ be the coordinate projection onto the first $m$ coordinates.
	There exists $(d,\rho,r) \in \mathcal E_n$ such that $\myedim_n(X_x) \leq (d,\rho,r)$ for every $x \in \pi(X)$ and $\myedim_n(X_x) = (d,\rho,r)$ for some $x \in \pi(X)$, where $X_x:=\{y \in F^n\;|\;(x,y) \in X\}$.
\end{proposition}
\begin{proof}
	We may assume that $X_x$ is bounded for every $x \in \pi(X)$ by applying the definable map $\varphi_n$ defined in Definition \ref{def:extended_dim_dmin}.
	Set $d=\max\{\dim X_x\;|\; x \in \pi(X)\}$.
	Recall that a linear order in $\Pi(n,d)$ is given in Definition \ref{def:extended_dim_dmin}.
	Let $\rho $ be a largest element in $\Pi(n,d)$ such that $\rho(X_x)$ has a nonempty interior for some $x \in \pi(X)$.
	Set $Y=\{x \in \pi(X)\;|\; \myint \rho(X_x) \neq \emptyset\}$.
	By the definition of $Y$, we have $\myedim_n(X_x) < (d,\rho, r)$ for every $x \in \pi(X) \setminus Y$ and $r>0$.
	
	We next consider the case in which $x \in Y$.
	By the definition of $d$ and Corollary \ref{cor:dim1}(2), $Z_x := \{y \in \rho(X_x)\;|\; \dim (X_x \cap \rho^{-1}(y))=0\}$ has a nonempty interior for every $x \in Y$.
	Apply Lemma \ref{lem:dmin_char} to the definable set $\{(x,y) \in Y\times F^n\;|\; y \in X_x \cap \rho^{-1}(Z_x)\}$.
	There exists a positive integer $r$ such that $\myrank (X_x \cap \rho^{-1}(y)) \leq r$ for every $y \in Z_x$ and $x \in Y$.  
	For each $x \in \pi(X)$, there exists $0<s \leq r$ such that $\{y \in Z_x\;|\; \myrank (X_x \cap \rho^{-1}(y)) =s\}$ has a nonempty interior by Proposition \ref{prop:dim}(2).
	The definable set $\{y \in \rho(X_x)\;|\; \myrank (X_x \cap \rho^{-1}(y)) =t\}$ is contained in $Z_x$ for each $t>0$.
	Therefore, we have $\myedim_n(X_x) \leq (d,\rho,r)$ for every $x \in Y$.
\end{proof}

We give a proof of Theorem \ref{thm:finite_pdeg_parametrized}.

\begin{proof}[Proof of Theorem \ref{thm:finite_pdeg_parametrized}]
	The `in particular' part is obvious once the former part of the theorem is proved.
	Set $X_x:=\{y \in F^n\;|\;(x,y) \in X\}$ for every $x \in \pi(X)$.
	We define $(M_{-1})_x, (M_0)_x$ and so on similarly.
	Note that the sequence $M_{-1} \cap \pi^{-1}(x), \ldots, M_{q} \cap \pi^{-1}(x)$ is a $p$-regular sequence of $X \cap \pi^{-1}(x)$ after removing empty sets other than $M_{-1} \cap \pi^{-1}(x)$ if and only if the sequence $(M_{-1})_x, \ldots, (M_{q})_x$ is a $p$-regular sequence of $X_x$ after removing empty sets other than $(M_{-1})_x$.
	We consider $X_x$ and $(M_i)_x$ rather than $X \cap \pi^{-1}(x)$ and $M_i \cap \pi^{-1}(x)$ for simplicity of notations.
	
	We prove the theorem by induction on $(d,\rho,r)=\max\{\myedim_n(X_x)\;|\; x \in \pi(X)\}$, which exists thanks to Proposition \ref{prop:uniform_bound_eRank}.
	We first consider the case in which $(d,\rho,r)$ is the smallest element in $\mathcal E_n$.
	We have $d=0$, $r=1$ and $\rho$ is the unique map from $F^n$ onto the singleton $F^0$.
	Fix a point $x \in \pi(X)$.
	The definable set $X_x$ is a definable $\mathcal C^p$ submanifold of $F^n$ for the same reason as Theorem \ref{thm:finite_pdeg}.
	The sequence $\emptyset, X_x$ is a $p$-regular sequence of $X_x$.
	We have only to put $q=0$, $M_{-1}=\emptyset$ and $M_0=X$ in this case.
	
	We next consider the other case.
	We may assume that $\rho$ is the projection onto the first $d$ coordinates by permuting the coordinates if necessary.
	We reduce to simpler cases step by step.
	We use the following claim in the reductions:
	\medskip
	
	\textbf{Claim.} Let $X$ be a nonempty definable subset of $F^{m+n}$.
	Set $(d,\rho,r)=\max\{\myedim_n(X_x)\;|\; x \in \pi(X)\}$.
	Let $Z$ be a definable subset of $F^{m+d}$ such that $\myedim_n(X_x \cap \rho^{-1}(Z_x))<(d,\rho,r)$ for each $x \in \pi(X)$, where $Z_x = \{y \in F^d\;|\;(x,y) \in Z\}$.
	The theorem holds for $X$ if the theorem holds for $U:=\bigcup_{x \in \pi(X)} (\{x\} \times (X_x \cap \rho^{-1}(F^d \setminus \mycl(Z_x))))$
	\begin{proof}[Proof of Claim]
		The equality $(X \setminus U)_x=X_x \cap \rho^{-1}(\mycl(Z_x))$ holds for each $x \in \pi(X)$ by the definition of $U$.
		Fix an arbitrary $x \in \pi(X)$.
		We can prove the inequality $\myedim_n((X \setminus U)_x)<(d,\rho,r)$ in the same manner as Theorem \ref{thm:finite_pdeg}.
		We omit the details.
		
		There exist finitely many definable sets $L_{-1},L_0,\ldots, L_{q_1}$ such that the sequence $(L_{-1})_x, \ldots, (L_{q_1})_x$ is a $p$-regular sequence of $(X \setminus U)_x$ for each $x \in \pi(X)$ after removing empty sets other than $(L_{-1})_x$ by the induction hypothesis.
		There exist finitely many definable sets $M_{-1},M_0,\ldots, M_{q_2}$ such that the sequence $(M_{-1})_x, \ldots, (M_{q_2})_x$ is a $p$-regular sequence of $U_x$ for each $x \in \pi(X)$ after removing empty sets other than $(M_{-1})_x$ by the assumption.
		Set $W=\{x \in \pi(X)\;|\; \dim (X \setminus U)_x <d\}$.
		Put 
		\begin{align*}
			& N_{-1}=\left(\bigcup_{x \in W} \{x\} \times (X \setminus U)_x\right) \cup \left(\bigcup_{x \in \pi(X) \setminus W} \{x\} \times (L_{-1})_x\right)\\
			&\qquad  \cup \left( \bigcup_{x \in \pi(X)} \{x\} \times \mycl_{X_x}((M_{-1})_x)\right)\\
			&N_i = \bigcup_{x \in \pi(X) \setminus W} \{x\} \times  ((L_i)_x \setminus \mycl_{X_x}((M_{-1})_x))\\
			& N_{q_1+1+j} =\bigcup_{x \in \pi(X) } \{x\} \times  ((M_j)_x \setminus \mycl_{X_x}((M_{-1})_x))
		\end{align*}
		for each $0 \leq i \leq q_1$ and $0 \leq  j \leq q_2$.
		Since $U_x$ is open in $X_x$, the sequence $(N_{-1})_x, (N_{0})_x, \ldots, (N_{q_1+q_2+1})_x$ is a $p$-partition sequence of $X_x$ after removing empty sets other than $(N_{-1})_x$.
	\end{proof}
	
	Set $V_1(x)= \{t \in \rho(X_x)\;|\; \dim(\rho^{-1}(t) \cap X_x)=0 \text{ and }\myrank(\rho^{-1}(t) \cap X_x)=r\}$ and $U_1 = \bigcup_{x \in \pi(X)} \{x\} \times (X_x \cap \rho^{-1}(\myint(V_1(x))))$.
	It is obvious that $\myedim_n((X \setminus U_1)_x)< \myedim_n(X_x)$ for each $x \in \pi(X)$ by the definition of extended rank.
	By Claim, we may assume that $\rho^{-1}(t) \cap X_x$ is of dimension $\leq 0$ and 	$\myrank(\rho^{-1}(t) \cap X_x)=r$ for each $x \in \pi(X)$ and $t \in \rho(X_x)$ by replacing $X$ with $U_1$.
	
	We set  $Y(x)=\bigcup_{t \in \rho(X_x)} \myIso(\pi^{-1}(t) \cap X_x)$ for each $x \in \pi(X)$.
	We next reduce to the case in which $Y(x)$ is open in $X_x$ for each $x \in \pi(X)$.
	We construct a definable subset $Z_2$ of $F^{m+d}$ satisfying the following conditions, where $U_2(x) = X_x \setminus \rho^{-1}((Z_2)_x)$:
	\begin{enumerate}
		\item[(a)] $Y'(x):=\bigcup_{t \in \rho(U_2(x))} \myIso(\pi^{-1}(t) \cap U_2(x))$ is open in $U_2(x)$ and 
		\item[(b)] $\myedim_n(X_x \cap \rho^{-1}((Z_2)_x))<(d,\pi,r)$
	\end{enumerate}
	for every $x \in \pi(X)$.
	
	Set $W_2=\{x \in \pi(X)\;|\; \myint(\rho(X_x \setminus Y(x)))=\emptyset\}$.
	We put $Z_2(x)=\rho(X_x \setminus Y(x))$ when $x \in W_2$.
	When $x \in \pi(X) \setminus W_2$, we consider the definable set 
	\begin{align*}
	V_2(x)&=\myint(\{t \in F^d\;|\; \dim (X_x \cap \rho^{-1}(t)) \leq 0 \text{ and }\\
	 & \quad \mycl(X_x \setminus Y(x)) \cap \rho^{-1}(t) = \mycl((X_x \setminus Y(x)) \cap \rho^{-1}(t))\}).
	\end{align*}
	It is dense by Lemma \ref{lem:dim1}.
	We put $Z_2(x)=X_x \setminus \rho^{-1}(V_2(x))$ when $x \in \pi(X) \setminus W_2$.
	We set $Z_2=\bigcup_{x \in \pi(X)} \{x\} \times Z_2(x)$.
	We can prove that $Z_2$ fulfills conditions (a) and (b) in the same manner as the proof of Theorem \ref{thm:finite_pdeg}.
	We omit the details.
	We may assume that $Y(x)$ is open in $X_x$ for each $x \in \pi(X)$ by replacing $X$ with $\bigcup_{x \in \pi(X)} \{x\} \times U_2(x)$ using Claim.

	For each $x \in \pi(X)$, let $M(x)$ be the set of points $y$ in $X_x$ such that there exists an open box $B$ in $F^n$ containing the point $y$ and the intersection $X_x \cap B$ is the graph of a definable $\mathcal C^p$ map defined on $\rho(B)$.
	We can prove that $M(x)$ is contained in $Y(x)$ and $\rho(Y(x) \setminus M(x))$ has an empty interior in the same manner as the proof of Theorem \ref{thm:finite_pdeg}.
	We omit the details.
	We can reduce to the case in which $Y(x)=M(x)$ by Claim by considering $\bigcup_{x \in \pi(X)} \{x\} \times (X_x \setminus \rho^{-1}(\mycl(\rho(Y(x)\setminus M(x)))))$ in place of $X$.
	It is also obvious that $Y(x)$ is open in $X_x$ after the replacement.
	
	The definable set $M(x)$ is a definable $\mathcal C^p$ submanifold of $F^n$ and $Y(x)=M(x)$ is open in $X_x$ for each $x \in \pi(X)$.
	It is obvious that $\myrank((X_x \setminus Y(x)) \cap \rho^{-1}(t))<r$ for every $t \in \rho(X_x)$ by the definition of $Y(x)$.
	This implies the inequality $\myedim_n(X_x \setminus M(x))<(d,\pi,r)$ because $Y(x)=M(x)$ by the assumption.
	There exists finitely many definable subsets $N'_{-1},N'_0,\ldots, N'_q$ such that 
	$(N'_{-1})_x,\ldots, (N'_q)_x$ is a $p$-regular sequence of $X_x \setminus M(x)$ for every $x \in \pi(X)$ after removing empty sets from the sequence other than $(N'_{-1})_x$.
	Set 
	\begin{align*}
		&W= \{x \in \pi(X)\;|\; \dim (X_x \setminus M(x))<\dim X_x\};\\
		&M=\bigcup_{x \in \pi(X)} \{x\} \times M(x);\\
		&N_{-1} = \bigcup_{x \in W}\{x\} \times (X_x \setminus M(x)) \cup \bigcup_{x \in \pi(X) \setminus W} \{x\} \times (N'_{-1})_x;\\
		&N_i= \bigcup_{x \in \pi(X) \setminus W} \{x\} \times (N'_i)_x \ (1 \leq i \leq q).
	\end{align*}
	The sequence $N_{-1}, N_0, \ldots, N_m, M$ is what we are looking for.
\end{proof}

As a corollary of Theorem \ref{thm:finite_pdeg_parametrized}, we give another characterization of d-minimal expansions of ordered fields.
\begin{corollary}
	Consider a definably complete expansion of an ordered field $\mathcal F=(F,<,+,\cdot,0,1,\ldots)$.
	The following are equivalent:
	\begin{enumerate}
		\item[(1)] $\mathcal F$ is d-minimal.
		\item[(2)]  Let $p$, $m$ and $n$ be positive integers.
		Let $\pi:F^{m+n} \to F^m$ be the coordinate projection onto the first $m$ coordinate.
		Let $X$ be a definable subset of $F^{m+n}$. 
		There exist finitely many definable subsets $M_{-1}, M_0, \ldots, M_q$ such that the sequence $M_{-1} \cap \pi^{-1}(x), \ldots, M_{q} \cap \pi^{-1}(x)$ is a $p$-regular sequence of $X \cap \pi^{-1}(x)$ for every $x \in \pi(X)$ after removing empty sets from the sequence other than $M_{-1} \cap \pi^{-1}(x)$.	
	\end{enumerate}
\end{corollary}
\begin{proof}
	The implication $(1) \Rightarrow (2)$ is Theorem \ref{thm:finite_pdeg_parametrized}.
	
	We consider the opposite implication $(2) \Rightarrow (1)$.
	Consider the case in which $n=1$.
	Let $X$ be a definable subset of $F^{m+1}$ and set $Y=\{(x,y) \in F^m \times F\;|\; (x,y) \in X, y \notin \myint_F(X_x)\}$, where $X_x=\{y \in F\;|\; (x,y) \in X\}$ as usual.
	There exist finitely many definable subsets $M_{-1}, M_0, \ldots, M_q$ such that the sequence $(M_{-1})_x, \ldots, (M_{q})_x $ is a $1$-regular sequence of $Y _x$ for every $x \in \pi(X)$ after removing empty sets from the sequence other than $(M_{-1})_x$.	
	
	We show that $X_x$ is a union of open set and at most $(q+1)$ many discrete sets for every $x \in \pi(X)$.
	This means that $\mathcal F$ is d-minimal.
	Fix an arbitrary point $x \in \pi(X)$.
	By the definition of $Y$, the difference $(X_x \setminus Y_x)$ is an open set.
	We have $\dim (M_{-1})_x <0$, which means that $(M_{-1})_x$ is an empty set.
	The definable set $(M_i)_x$ is either an empty set or a definable $\mathcal C^1$ submanifold of dimension zero.
	It is obvious that $(M_i)_x$ is discrete in the latter case by the definition of definable $\mathcal C^1$ submanifolds.
	We have shown that $X_x$ is a union of open set and at most $(q+1)$ many discrete sets because $X_x=(X_x \setminus Y_x) \cup  (M_0)_x \cup \cdots \cup (M_{q})_x$. 
\end{proof}

\subsection{Basic properties of partition degree}\label{subsec:basic_partition}
We introduce several basic properties of partition degree in this subsection.
We first introduce the following basic inequality:
\begin{lemma}\label{lem:partition_included}
	Consider a d-minimal expansion of an ordered field $\mathcal F=(F,<,+,\cdot,0,1,\ldots)$.
	Let $X \subseteq Y$ be definable sets of the same dimension.
	Then the inequality $\mypdeg(X) \leq \mypdeg(Y)$ holds.
\end{lemma}
\begin{proof}
	Set $d=\dim X=\dim Y$ and $m=\mypdeg(Y)$.
	Let $r$ be a positive integer.
	Let $N_{-1}, N_0,\ldots, N_m$ be an $r$-partition sequence of $Y$.
	Set $M'_i=\myint_{N_i}(N_i \cap X)$ and $Z_i=(N_i \cap X) \setminus M'_i$ for $0 \leq i \leq m$.
	We have $\dim Z_i<d$ by Corollary \ref{cor:mfd_bdry}.
	Set $M_{-1}=\mycl_X(N_{-1} \cup \bigcup_{i=0}^m Z_i)$.
	It is of dimension smaller than $d$ by Proposition \ref{prop:dim}(2) and Lemma \ref{lem:dim4}.
	For $0 \leq i \leq m$, the definable set $M_i:=M'_i \setminus M_{-1}$ is a definable $\mathcal C^r$ submanifold of dimension $d$ by Lemma \ref{lem:open_mfd} when $M_i$ is not empty.
	The sequence $M_{-1}, M_0, \ldots, M_m$ is an $r$-partition sequence of $X$ after removing empty sets other than $M_{-1}$ from the sequence.
	It means that $\mypdeg(X) \leq \mypdeg(Y)$.
\end{proof}

We also need the following lemma:
\begin{lemma}\label{lem:small_gap}
	Consider a d-minimal expansion of an ordered field $\mathcal F=(F,<,+,\cdot,0,1,\ldots)$.
	Let $X$ and $Y$ be definable subsets of $F^n$ of the same dimension.
	If $\dim ((X \setminus Y) \cup (Y \setminus X)) < \dim X$, the equality $\mypdeg(X)=\mypdeg(Y)$ holds.
\end{lemma}
\begin{proof}
	Set $d=\dim X=\dim Y$.
	Fix a positive integer $r$.
	By symmetry, we have only to show $\mypdeg(Y) \leq \mypdeg(X)$.
	Set $m= \mypdeg(X)$.
	There exists an $r$-partition sequence $(M_{-1}, M_0,\ldots, M_m)$ of $X$ by the definition of partition degree.
	Set $N_i = M_i \setminus \mycl((X \setminus Y) \cup (Y \setminus X) )$ for $0 \leq i \leq m$.
	We have $\dim (\mycl((X \setminus Y) \cup (Y \setminus X) ))<d$ by the assumption and Lemma \ref{lem:dim4}.
	It implies that $\dim N_i =d$ by Proposition \ref{prop:dim}(2).
	Set $N_{-1}:=Y \cap (M_{-1} \cup \mycl((X \setminus Y) \cup (Y \setminus X) ))$.
	We have $\dim N_{-1}<d$ by Proposition \ref{prop:dim}(2).
	It is easy to check that $(N_{-1},N_0,\ldots, N_m)$ is an $r$-partition sequence of $Y$.
	It means that $\mypdeg(Y) \leq m$.
\end{proof}

The following proposition gives a procedure to generate an $r$-partition sequence of a definable set of minimum length.
\begin{proposition}\label{prop:regular_seq}
	Consider a d-minimal expansion of an ordered field $\mathcal F=(F,<,+,\cdot,0,1,\ldots)$.
	Let $r$ be a positive integer and $X$ be a definable subset of $F^n$.
	We define $X\langle i \rangle$ as follows for each $i \geq -1$:
	\begin{itemize}
		\item $X\langle -1 \rangle = X$;
		\item $X\langle i \rangle = X \langle i-1 \rangle \setminus \myReg_r(X \langle i-1 \rangle ) $.
	\end{itemize}
	There exists a nonnegative integer $m$ such that $\dim X \langle m \rangle < \dim X$ and the equality $\mypdeg(X)=m$ holds.
	In addition, $(X\langle m \rangle, \myReg_r^m(X),\ldots, \myReg_r^0(X))$ is an $r$-partition sequence of $X$, where $\myReg_r^i(X):=\myReg_r(X\langle i-1 \rangle)$ for $0 \leq i \leq m$.
\end{proposition}
\begin{proof}
	Let $M_{-1},M_0,\ldots, M_m$ be an $r$-partition sequence of minimum length.
	It is obvious that $M_m \subseteq \myReg_r(X)$ by the definition of the set $\myReg_r(X)$, Lemma \ref{lem:open_regular} and Corollary \ref{cor:regular2}.
	The set $\myReg_r(X)$ is definable and open by Corollary \ref{cor:regular1}.
	It is a definable $\mathcal C^r$ submanifold of $F^n$ by Corollary \ref{cor:regular2}.
	
	We prove the proposition by induction on $m$.
	When $m=0$, we have $X \langle 0 \rangle \subseteq M_{-1}$.
	It implies that $\dim X \langle 0 \rangle  < \dim X$.
	The sequence $(X \langle 0 \rangle, \myReg_r(X))$ is an $r$-partition sequence by Lemma \ref{lem:open_mfd}, Corollary \ref{cor:regular1} and Corollary \ref{cor:regular2}.
	The proposition holds in this case.
	
	We next consider the case in which $m>0$.
	We have $\mypdeg(X \setminus M_m) = m-1$ by the definition of partition degree.
	Since $X \setminus \myReg_r(X) \subseteq X \setminus M_m$, the inequality $\mypdeg (X \setminus \myReg_r(X)) \leq m-1$ holds by Lemma \ref{lem:partition_included}.
	If $\mypdeg (X \setminus \myReg_r(X)) < m-1$, we can construct an $r$-partition sequence of $X \setminus \myReg_r(X)$ of length smaller than $m+1$.
	We can construct an $r$-partition sequence of $X$ of length smaller than $m+2$ by concatenating the sequence of length smaller than $m+1$ with $\myReg_r(X)$.
	It contradicts the assumption that $\mypdeg(X)=m$.
	We have demonstrated that $\mypdeg (X \langle 0 \rangle) =m-1$.
	Set $Y=X \langle 0 \rangle = X \setminus \myReg_r(X)$.
	We obviously have $Y \langle m-1 \rangle = X \langle m \rangle$ and $\myReg_r^j(Y)=\myReg_r^{j+1}(X)$ for $0 \leq j \leq m-1$.
	Apply the induction hypothesis to $Y$, we get an $r$-partition sequence $(Y\langle m-1 \rangle, \myReg_r^{m-1}(Y), \ldots, \myReg_r^0(Y), \myReg_r(X))= (X\langle m \rangle, \myReg_r^m(X), \ldots, \myReg_r^0(X))$ of $X$ of minimum length.
\end{proof}

Dimension is preserved under definable bijection, but partition degree is not.
We can say more.
We can always find a definable set of partition degree zero definably bijective to a given nonempty definable set.
\begin{proposition}\label{prop:always_pdeg_zero}
	Consider a d-minimal expansion of an ordered field $\mathcal F=(F,<,+,\cdot,0,1,\ldots)$.
	Let $X$ be a nonempty definable subset of $F^n$.
	There exist a definable subset $Y$ of $F^{n+1}$ of $\mypdeg(Y)=0$ and a definable bijection $\varphi:X \to Y$.
\end{proposition}
\begin{proof}
	Fix a positive integer $r$.
	Set $m=\mypdeg(X)$.
	Set $Y=X\langle m \rangle \times \{-1\} \cup \bigcup_{i=0}^m \myReg_r^i(X) \times \{i\}$.
	We can define a natural definable bijection between $X$ and $Y$.
	It is obvious that $\myReg_r(Y)=  \bigcup_{i=0}^m \myReg_r^i(X) \times \{i\}$.
	It means that $\mypdeg(Y)=0$.
\end{proof}

Partition degree is preserved under definable homeomorphism:
\begin{theorem}\label{thm:homeo_inv}
	Consider a d-minimal expansion of an ordered field $\mathcal F=(F,<,+,\cdot,0,1,\ldots)$.
	Let $X$ and $Y$ be definable sets which are definably homeomorphic to each other.
	Then the equality $\mypdeg(X)=\mypdeg(Y)$ holds.
	In addition, $\varphi(\myReg_r(X)) \setminus \myReg_r(Y)$ is of dimension smaller than $\dim Y$, where $r$ is a positive integer and $\varphi:X \to Y$  is a definable homeomorphism.
\end{theorem}
\begin{proof}
	Note that $\dim X=\dim Y$ by Corollary \ref{cor:dim1}(1).
	Set $d=\dim X$.
	Let $F^m$ and $F^n$ be the ambient spaces of $X$ and $Y$, respectively.
	Fix a positive integer $r$.
	Let $\varphi:X \to Y$ be a definable homeomorphism and $\psi:Y \to X$ be its inverse.
	
	We first show the `in addition' part of the theorem; that is,  $\varphi(\myReg_r(X)) \setminus \myReg_r(Y)$ is of dimension smaller than $d$.
	Set $Z=\varphi(\myReg_r(X)) \setminus \myReg_r(Y)$ and assume for contradiction that $\dim Z=d$.
	We have $Z \cap \myReg_r(Y)=\emptyset$.
	Since $\varphi$ is a definable bijection, we have $\dim \varphi^{-1}(Z)=d$ by Corollary \ref{cor:dim1}(1).
	Since $\myReg_r(X)$ is a definable $\mathcal C^r$ submanifold of dimension $d$ by Corollary \ref{cor:regular2}, the definable set $\varphi^{-1}(Z)$ has a nonempty interior in $\myReg_r(X)$ by Corollary \ref{cor:mfd_same_dim}.
	Let $U$ be a nonempty definable open subset of $\myReg_r(X)$ contained in $\varphi^{-1}(Z)$.
	The definable set $U$ is open in $X$ because $\myReg_r(X)$ is open in $X$ by Corollary \ref{cor:regular1}.
	Set $V=\varphi(U)$.
	It is open in $Y$ because $\varphi$ is a homeomorphism and we have $\dim V=d$ by Corollary \ref{cor:dim1}(1).
	Consider the definable set $\myReg_r(V)$.
	Any point in $\myReg_r(V)$ is regular in $Y$ by Lemma \ref{lem:open_regular} because $V$ is open in $Y$ and $\dim V=d$.
	It means that $\myReg_r(V) \subseteq \myReg_r(Y)$.
	However,  $\myReg_r(V)$ is contained in $Z$ and $Z$ has an empty intersection with $\myReg_r(Y)$, which is absurd.
	We have proven that $\dim Z<d$.
	The `in addition' part has been proved.
	
	Set $Z'=\psi(\myReg_r(Y)) \setminus \myReg_r(X)$.
	By symmetry, we have $\dim Z'<d$.
	By Lemma \ref{lem:dim4}, we have $\dim \mycl_X(Z)<d$ and $\dim \mycl_Y(Z') <d$.
	Set $X'=X \setminus (\mycl_X(Z) \cup \psi(\mycl_Y(Z')))$ and  $Y'=Y \setminus (\mycl_Y(Z') \cup \varphi(\mycl_X(Z)))$.
	We have $\dim X'=\dim Y'=d$ by Proposition \ref{prop:dim}(2) and Corollary \ref{cor:dim1}.
	The restriction of $\varphi$ to $X'$ gives a definable homeomorphism onto its image $Y'$.
	We also have $\varphi(\myReg_r(X'))=\myReg_r(Y')$ because the equalities $\myReg_r(X') = \myReg_r(X) \setminus  (\mycl_X(Z) \cup \psi(\mycl_Y(Z'))$ and $\myReg_r(Y')=\myReg_r(Y) \setminus (\mycl_Y(Z') \cup \varphi(\mycl_X(Z)))$ hold by Lemma \ref{lem:open_regular}.
	We also have $\mypdeg(X)=\mypdeg(X')$ and $\mypdeg(Y)=\mypdeg(Y')$ by  Lemma \ref{lem:small_gap}.
	Therefore, we may assume that $\varphi(\myReg_r(X))=\myReg_r(Y)$ by replacing $X$ and $Y$ with $X'$ and $Y'$, respectively.
	
	We prove the theorem by induction on $m=\mypdeg(X)$.
	When $m=0$, we have $\dim X \setminus \myReg_r(X)<d$ by Proposition \ref{prop:regular_seq}.
	Because $\varphi$ is a bijection and $\varphi(\myReg_r(X))=\myReg_r(Y)$, we get $\dim Y \setminus \myReg_r(Y)<d$ by Corollary \ref{cor:dim1}(1).
	This implies that $\mypdeg(Y)=0$ by Proposition \ref{prop:regular_seq}.
	
	When $m>0$, we have $\mypdeg (X \setminus \myReg_r(X)) = \mypdeg(X)-1$ and $\mypdeg (Y \setminus \myReg_r(Y)) = \mypdeg(Y)-1$ by Proposition \ref{prop:regular_seq}.
	We have $\mypdeg (X \setminus \myReg_r(X)) =\mypdeg (Y \setminus \myReg_r(Y)) $ by the induction hypothesis because there exists a definable homeomorphism between $X \setminus \myReg_r(X)$ and $Y \setminus \myReg_r(Y)$ by the assumption.
	We obtain $\mypdeg(Y)=\mypdeg(X)$.
\end{proof}

The following simple formula is not used in this paper, but it is worth mentioning.
\begin{theorem}\label{thm:times_partition_degree}
Consider a d-minimal expansion of an ordered field.
Let $X$ and $Y$ be definable sets.
Then the equality $\mypdeg(X \times Y)=\mypdeg(X)+\mypdeg(Y)$ holds.
\end{theorem}
\begin{proof}
	Set $Z=X \times Y$, $d=\mypdeg(X)$ and $e=\mypdeg(Y)$.
	Fix a positive integer $r$.
	We set $X_i := \myReg_r^i(X)$, $Y_j:= \myReg_r^j(Y)$ and $Z_k:=\myReg_r^k(Z)$ using the notations in Proposition \ref{prop:regular_seq} for $0 \leq i \leq d$, $0 \leq j \leq e$ and $0 \leq k \leq d+e$.
	Set $X'=X \setminus (\bigcup_{0 \leq i \leq d} X_i)$ and $Y'=Y \setminus (\bigcup_{0 \leq j \leq e} Y_j)$.
	We prove the equality 
	\begin{equation}\label{eq:zxy}
	Z_k= \bigcup_{0 \leq i \leq d, 0 \leq j \leq e, i+j=k}X_i \times Y_j
	\end{equation}
	for each $0 \leq k \leq d+e$. 
	This equality implies the theorem as follows:
	Assume that equality (\ref{eq:zxy}) holds.
	Set 
	$Z':=Z \setminus (\bigcup_{0 \leq k\leq d+e} Z_k) = (X' \times Y) \cup (X \times Y')$.
	We get $\dim Z' < \dim X + \dim Y$ by Proposition \ref{prop:dim} and Corollary \ref{cor:dim1}(2).
	We also have $\dim Z_k = \dim X + \dim Y$ for each $0 \leq k \leq d+e$ by Proposition \ref{prop:dim} and Corollary \ref{cor:dim1}(2).
	This implies that $\mypdeg(X \times Y)=d+e$ by Proposition \ref{prop:regular_seq}.
	The remaining task is to prove equality (\ref{eq:zxy}).
	
	Let us begin to prove equality (\ref{eq:zxy}) by induction on $k$.
	When $k=0$, the equality follows from Corollary \ref{cor:product_reg}.
	We consider the case in which $k>0$.
	Put $$Z'_k = \bigcup_{0 \leq i \leq d, 0 \leq j \leq e, i+j=k}X_i \times Y_j $$ for the simplicity of notations.
	We first demonstrate the inclusion $Z'_k\subseteq Z_k$.
	Set $$W_k=X \times Y \setminus \bigcup_{0 \leq i \leq d, 0 \leq j \leq e, i+j < k}X_i \times Y_j.$$
	The set $W_k$ coincides with $Z\langle k-1 \rangle$ in the notation of Proposition \ref{prop:regular_seq} by the induction hypothesis. 
	Fix an arbitrary $0 \leq i \leq d$ and $0 \leq j \leq e$ such that $i+j=k$.
	We show that $X_i \times Y_j$ is open in $W_k$.
	Since  $X_i$ is open in $X \setminus \bigcup_{0 \leq i' < i}X_{i'}$ by Proposition \ref{prop:regular_seq}, there exists an open set $U$ such that $U \cap (X \setminus \bigcup_{0 \leq i' < i}X_{i'}) = X_i$.
	There exists an open set $V$ such that $V \cap (Y \setminus \bigcup_{0 \leq j' < j}Y_{j'}) = Y_j$ for the same reason.
	It is easy to show that $(U \times V) \cap W_k = X_i \times Y_j$.
	This means that $X_i \times Y_j$ is open in $W_k$.
	
	It is obvious that $X_i \times Y_j$ is a definable $\mathcal C^r$ submanifold.
	Any point in $X_i \times Y_j$ is $r$-regular in $X_i \times Y_j$ by Corollary \ref{cor:regular2}.
	Since $X_i \times Y_j$ is open in $W_k$ and $\dim (X_i \times Y_j)=\dim W_k=\dim (X \times Y)$, any point in $X_i \times Y_j$ is $r$-regular in $W_k$ by Lemma \ref{lem:open_regular}.
	This means that $X_i \times Y_j \subseteq Z_k$.
	We have demonstrated the inclusion $Z'_k \subseteq Z_k$.
	
	We next prove the opposite inclusion.
	Take a point $(a,b) \in Z_k$.
	By Corollary \ref{cor:product_reg2}, the point $b$ is $r$-regular in $(W_k)_a:=\{y \in Y\;|\; (a,y) \in W_k\}$.
	The set $(W_k)_a$ is one of the forms $Y \langle i \rangle$ for $-1 \leq i \leq e$ in the notation of Proposition \ref{prop:regular_seq} by the definition of $W_k$.
	We get $(W_k)_a \neq Y \langle e \rangle$ by Corollary \ref{cor:product_reg2} because $\dim Y \langle e \rangle < \dim Y$.
	We get $b \in Y_j$ for some $0 \leq j \leq e$.
	By symmetry, we also get $a \in X_i$ for some $0 \leq i \leq d$.
	
	Fix $0 \leq i \leq d$ so that $a \in X_i$.
	We consider two separate cases.
	The first case is the case in which $e \leq k-i$.
	We have $(W_k)_a=Y \langle e \rangle$ in this case, which is absurd.
	%
	The remaining case is the case in which $e > k-i$.
	We have $(W_k)_a=Y \langle k - i \rangle$.
	Since $b$ is $r$-regular in $(W_k)_a$, we have $b \in Y_{k-i}$.
	It implies that $(a,b) \in X_i \times Y_{k-i} \subseteq Z'_k$.
	We have demonstrated the inclusion $Z_k \subseteq Z'_k$.
\end{proof}

The following proposition is also worth mentioning.
\begin{proposition}\label{prop:wide_cr}
	Let $r$ be a positive integer.
	Consider a d-minimal expansion of an ordered field whose universe is $F$.
	Let $X$ be a definable set and $f_1,\ldots, f_k:X \to F$ be definable functions.
	There exists a definable open subset $U$ of $X$ such that at least one of the inequalities $\dim X  \setminus U < \dim X$ and $\mypdeg X \setminus U < \mypdeg X$ holds, $U$ is a definable $\mathcal C^r$ submanifold and the function $f_i$ restricted to $U$ is of class $\mathcal C^r$ for each $1 \leq i \leq k$.
\end{proposition}
\begin{proof}
	Recall that $\myReg_r(X)$ is a definable $\mathcal C^r$ submanifold by Corollary \ref{cor:regular2}.
	By Lemma \ref{lem:mfd_cr} and Proposition \ref{prop:dim}(2), there exists a definable open subset $U$ of $\myReg_r(X)$ such that the function $f_i$ restricted to $U$ is of class $\mathcal C^r$ for each $1 \leq i \leq k$ and $\dim \myReg_r(X) \setminus U < \dim \myReg_r(X)$.
	The definable set $U$ is a definable $\mathcal C^r$ submanifold by Lemma \ref{lem:open_mfd}.
	The definable set $U$ is open in $X$ and $\dim \myReg_r(X) \setminus U < \dim X$ because $\myReg_r(X)$ is open in $X$ and $\dim \myReg_r(X)=\dim X$ by Corollary \ref{cor:regular1}.
	At least one of the inequalities $\dim X  \setminus U < \dim X$ or $\mypdeg X \setminus U < \mypdeg X$ holds by Proposition \ref{prop:regular_seq}.
\end{proof}

\subsection{Hugeness and leanness}\label{subsec:huge_lean}
The notion of $d$-largeness is used in the study of definable groups by Wencel \cite{Wencel}, but this notion is not so useful for our purpose in d-minimal context.
We need alternative notions.
We introduce the notions of leanness, hugeness and largeness.
\begin{definition}
	Consider an expansion of a dense linear order without endpoints.
	Let $X$ and $Y$ be definable sets with $X \subseteq Y$.
	We say that $X$ is \textit{lean} in $Y$ if and only if $\dim(\myint_YX)<\dim Y$.
	We call that $X$ is \textit{huge} in $Y$ if $Y \setminus X$ is lean in $Y$.
	
	Let $r$ be a positive integer.
	Consider a definably complete expansion of an ordered field.
	Let $X$ and $Y$ be definable sets with $X \subseteq Y$.
	We say that $X$ is \textit{$r$-large} in $Y$ if $\myint_Y(X \cap \myReg_r(Y)) \neq \emptyset$.
\end{definition}

The following equivalence is useful:
\begin{proposition}\label{prop:lean_equiv}
	Consider a d-minimal expansion of an ordered field.
	Let $X$ and $Y$ be definable sets with $X \subseteq Y$.
	Let $r$ be a positive integer.
	The following are equivalent:
	\begin{enumerate}
		\item[(1)] The definable set $X$ is lean in $Y$;
		\item[(2)] The definable set $X$ is not $r$-large in $Y$;
		\item[(3)] $\dim (X \cap \myReg_r(Y)) < \dim Y$.
	\end{enumerate}
	In addition, conditions (1) through (3) hold when $\dim X<\dim Y$.
\end{proposition}
\begin{proof}
	Note that $\myReg_r(Y)$ is open in $Y$ by Corollary \ref{cor:regular1}.
	Therefore, the equality $\myint_{\myReg_r(Y)}(X \cap \myReg_r(Y)) = \myint_{Y}(X \cap \myReg_r(Y))$ holds.
	The set $\myReg_r(Y)$ is a definable $\mathcal C^r$ submanifold of dimension $d$ by Corollary \ref{cor:regular2}.
	We use these facts without notice in the proof.
	
	We first consider the case in which $\dim X<\dim Y$.
	Conditions (1) through (3) hold in this case. 
	In fact, we have $\dim (\myint_YX) \leq \dim X<\dim Y$.
	We also have $\dim (X \cap \myReg_r(Y))  \leq \dim X < \dim Y$.
	It is equivalent to the condition that $\myint_Y(X \cap \myReg_r(Y)) = \myint_{\myReg_r(Y)}(X \cap \myReg_r(Y)) = \emptyset$ by Corollary \ref{cor:mfd_same_dim}.
	It implies that $X$ is not $r$-large in $Y$.

	We concentrate on the case in which $\dim X=\dim Y$ in the rest of the proof.
	Set $d=\dim Y$.
	We prove that condition (1) implies condition (2).
	We have $\dim \myint_{\myReg_r(Y)} (X \cap \myReg_r(Y))  = \dim \myint_Y  (X \cap \myReg_r(Y)) \leq \dim (\myint_YX) <d$ by the assumption.
	It implies that $\myint_Y  (X \cap \myReg_r(Y)) = \myint_{\myReg_r(Y)} (X \cap \myReg_r(Y))  = \emptyset$ by Corollary \ref{cor:mfd_same_dim}.
	
	We next show $(2) \Rightarrow (3)$.
	We have $\myint_{\myReg_r(Y)}(X \cap \myReg_r(Y)) =\myint_Y(X \cap \myReg_r(Y))=\emptyset$.
	It means $\dim (X \cap \myReg_r(Y))<d$ by Corollary \ref{cor:mfd_same_dim}.
	
	We finally prove $(3) \Rightarrow (1)$.
	Set $m=\mypdeg(Y)$ and $V=\myint_YX$.
	The sequence $$(Y\langle m \rangle, \myReg_r^m(Y), \ldots, \myReg_r^0(Y))$$ is an $r$-partition sequence by Proposition \ref{prop:regular_seq}.
	Assume for contradiction that $\dim V= d$.
	It is obvious that $\dim (V \cap Y\langle m \rangle)<d$ and $\dim V \cap \myReg_r^0(Y) \leq \dim (X \cap \myReg_r(Y)) < d$.
	There exists $1 \leq i \leq m$ such that $\dim (V \cap \myReg_r^i(Y))=d$ by Proposition \ref{prop:dim}(2) and the definition of $r$-partition sequences.
	Set $i_0:= \min\{1 \leq i \leq m\;|\; \dim (V \cap \myReg_r^i(Y))=d\}$.
	The set $$W:= V \cap \myReg_r^{i_0}(Y) \setminus \left(\bigcup_{j=0}^{i_0-1} \mycl (V \cap \myReg_r^j(Y))\right)$$
	is of dimension $d$ by Lemma \ref{lem:dim4} and Proposition \ref{prop:dim}(2).
	We prove that the set $W$ is open in $V$.
	The set $W$ is contained in $Y \langle i_0 \rangle = Y \setminus \left(\bigcup_{j=0}^{i_0-1} \myReg_r^j(Y)\right)$.
	The set $\myReg_r^{i_0}(Y)$ is open in $Y \langle i_0 \rangle$ by Corollary \ref{cor:regular1}.
	The intersection $V \cap \myReg_r^{i_0}(Y)$ is open in $V \cap Y\langle i_0 \rangle$.
	On the other hand, the difference $V':=V \setminus \left(\bigcup_{j=0}^{i_0-1} \mycl (V \cap \myReg_r^j(Y))\right)$ is open in $V$ and it is contained in $V \cap Y\langle i_0 \rangle$.
	The set $W$ is open in $V$ because it is the intersection of $V \cap \myReg_r^{i_0}(Y)$ with $V'$.
	
	There exists $x \in W$ such that $\dim W \cap B=d$ for any open box $B$ containing the point $x$ by  Corollary \ref{cor:dim2}.
	Since $W$ is an open subset of $V=\myint_YX$, we have $Y \cap B= X \cap B \subseteq W$ if we take a sufficiently small open box $B$ containing the point $x$.
	By the construction of $W$ and $B$, the intersection of $B$ with $\myReg_r^j(Y)$ is empty for each $0 \leq j < i_0$.
	It implies that $Y \cap B= Y \langle i_0 \rangle \cap B$.
	Since $\myReg_r(Y\langle i_0 \rangle) = \myReg_r^{i_0}(Y)$, the point $x$ is $r$-regular in $Y\langle i_0 \rangle$.
	The equality $Y \cap B= Y \langle i_0 \rangle \cap B$ implies that the point $x$ is $r$-regular in $Y$ by Lemma \ref{lem:open_regular}.
	This contradicts the definition of $\myReg_r^{i_0}(Y)$.
\end{proof}

\begin{corollary}\label{cor:lean_equiv}
	Consider a d-minimal expansion of an ordered field.
	Let $X \subseteq Y$ be definable sets.
	Assume that $\mypdeg(Y)=0$.
	The definable set $X$ is lean in $Y$ if and only if $\dim X<\dim Y$.
\end{corollary}
\begin{proof}
	Fix a positive integer $r$.
	We have $ \dim X= \max\{\dim (X \cap \myReg_r(Y)),\dim(X \setminus \myReg_r(Y))\} \leq \max\{\dim(X \cap \myReg_r(Y)),\dim(Y \setminus \myReg_r(Y))\}$ by Proposition \ref{prop:dim}(2).
	On the other hand, we have $\dim Y \setminus \myReg_r(Y)<\dim Y$ by Proposition \ref{prop:regular_seq}.
	This implies that $\dim X<\dim Y$ if and only if $\dim (X \cap \myReg_r(Y))<\dim Y$.
	This implies the corollary by Proposition \ref{prop:lean_equiv}.
\end{proof}

\begin{remark}
	A definable subset $X$ of $Y$ is called \textit{$d$-large} in \cite{Wencel} if $\dim (Y \setminus X) < \dim Y$.
	Corollary \ref{cor:lean_equiv} implies that the notion of hugeness coincides with the notion of $d$-largeness when the structure is a definably complete locally o-minimal expansion of an ordered field.
	In fact, we always have $\mypdeg(X)=0$ for each nonempty definable set $X$ in this setting.
	We can decompose $X$ into finitely many definable $\mathcal C^r$ submanifolds $C_1, \ldots, C_m$ so that the frontier of $C_i$ is the union of members in a subfamily of $\{C_1, \ldots, C_m\}$ for each $1 \leq i \leq m$.
	See \cite[Theorem 2.11]{FK}.
	See also \cite[Theorem 5.6]{Fornasiero0} and \cite[Proposition 2.11]{FKK}.
	Let $M_0$ be the union of all $C_i$ with $\dim C_i=\dim X$ and let $M_{-1}$ be the union of all $C_i$ with $\dim C_i<\dim X$.
	The sequence $M_{-1},M_0$ is an $r$-partition sequence, which implies that $\mypdeg(X)=0$.
\end{remark}

We prove basic properties of hugeness.

\begin{lemma}\label{lean:product}
Consider a d-minimal expansion of an ordered field.
Let $X_i$ and $Y_i$ be definable sets with $X_i \subseteq Y_i$ for $i=1,2$.
If $X_1$ and $X_2$ are huge in $Y_1$ and $Y_2$, respectively, the Cartesian product $X_1 \times X_2$ is huge in $Y_1 \times Y_2$.
\end{lemma}
\begin{proof}
	Let $r$ be a positive integer.
	We have $\myReg_r(Y_1 \times Y_2) \setminus (X_1 \times X_2) = (\myReg_r(Y_1) \times (\myReg_r(Y_2) \setminus X_2)) \cup ((\myReg_r(Y_1) \setminus X_1) \times \myReg_r(X_2))$ by Corollary \ref{cor:product_reg}.
	We have $\dim \myReg_r(Y_1 \times Y_2) \setminus (X_1 \times X_2) <\dim X_1 + \dim X_2 = \dim (X_1 \times X_2)$ by the assumption, Proposition \ref{prop:lean_equiv}, Proposition \ref{prop:dim}(2) and Corollary \ref{cor:dim1}(2).
	This inequality means that $X_1 \times X_2$ is huge in $Y_1 \times Y_2$ by Proposition \ref{prop:lean_equiv}.
\end{proof}

\begin{lemma}\label{lem:lean_homeo}
	Consider a d-minimal expansion of an ordered field.
	Let $\varphi:X \to Y$ be a definable homeomorphism.
	A definable subset $S$ of $X$ is lean in $X$ if and only if $\varphi(S)$ is lean in $Y$.
	Equivalently, $S$ is huge in $X$ if and only if $\varphi(S)$ is huge in $Y$.
\end{lemma}
\begin{proof}
	We obviously have $\varphi(\myint_XS)=\myint_Y\varphi(S)$ because $\varphi$ is a homeomorphism.
	We get $\dim (\myint_XS) < \dim X$ if and only if $\dim (\myint_Y\varphi(S))<\dim Y$ by Corollary \ref{cor:dim1}(1).
\end{proof}

\begin{lemma}\label{lem:lean_interior}
Consider a d-minimal expansion of an ordered field.
Let $X$ and $Y$ be definable sets with $X \subseteq Y$.
The definable set $X$ is huge in $Y$ if and only if $\myint_{Y}(X)$ is huge in $Y$. 
\end{lemma}
\begin{proof}
	The `if' part is obvious from the definition of hugeness.
	We prove the `only if' part.
	Let $r$ be a positive integer.
	Set $Z= \myReg_r(Y) \setminus X $.
	We have $\dim Z<\dim Y$ by Proposition \ref{prop:lean_equiv}.
	We have $\dim (\myReg_r(Y) \cap\partial_Y(Z))<\dim Y$.
	In fact, this follows from Corollary \ref{cor:mfd_same_dim} and Lemma \ref{lem:dim4} when $\myint_{ \myReg_r(Y)}(Z)=\emptyset$.
	In the other case, we get this inequality by Corollary \ref{cor:regular1} , Corollary \ref{cor:regular2} and Corollary \ref{cor:mfd_bdry}.
	We have $\myReg_r(Y) \setminus \myint_Y(X) \subseteq  Z \cup (\myReg_r(Y) \cap\partial_Y (Z))$.
	We prove this inclusion.
	Let $x \in \myReg_r(Y) \setminus \myint_Y(X)$.
	It is obvious $x \in Z$ when $x \notin X$.
	When $x \in X$, we obviously have $x \notin Z$.
	If the condition $x \notin \mycl_Y(Z)$ holds, then we have $x \in \myint_Y(X)=\myint_{ \myReg_r(Y)}(X)$ because $\myReg_r(Y)$ is open in $Y$ by Corollary \ref{cor:regular1}.
	This is a contradiction. 
	We have demonstrated that $x \in \partial _Y(Z)$.
	We get the inequality $\dim \myReg_r(Y) \setminus \myint_Y(X) <\dim Y$ by Proposition \ref{prop:dim}(2), which implies that $\myint_{Y}(X)$ is huge in $Y$ by Proposition \ref{prop:lean_equiv}. 
\end{proof}

\begin{lemma}\label{lem:large_equiv}
Let $r$ be a positive integer.
Consider a d-minimal expansion of an ordered field.
Let $X$ and $Y$ be definable subsets of $F^n$ with $X \subseteq Y$ and $\dim Y=d$.
The definable set $X$ is $r$-large in $Y$ if and only if there exist an open box $U$ in $F^n$, a coordinate projection $\pi:F^n \to F^d$ and a definable $\mathcal C^r$ map $\varphi: \pi(U) \to F^{n-d}$ such that $U \cap X = U \cap Y$ and $U \cap Y$ is a permutated graph of $\varphi$ under $\pi$. 
\end{lemma}
\begin{proof}
	We first prove the `if' part.
	Since $U \cap Y$ is a permutated graph of a definable $\mathcal C^r$ map, it is contained in $\myReg_r(Y)$ by Lemma \ref{lem:regular}.
	The equality $U \cap X=U \cap Y$ implies that $\myint_{\myReg_r(Y)}(X \cap \myReg_r(Y)) \neq \emptyset$.
	Since $\myReg_r(Y)$ is open in $Y$ by Corollary \ref{cor:regular1}, we have $\myint_{Y}(X \cap \myReg_r(Y)) \neq \emptyset$.
	
	We next prove the `only if' part.
	Note that $\myint_{Y}(X \cap \myReg_r(Y)) = \myint_{\myReg_r(Y)}(X \cap \myReg_r(Y)) $ because $\myReg_r(Y)$ is open in $Y$.
	Take a point $x \in \myint_{\myReg_r(Y)}(X \cap \myReg_r(Y))$.
	We have $U \cap X=U \cap \myReg_r(Y)=U \cap Y$ for any sufficiently small open box $U$ in $F^n$ containing the point $x$ because $\myReg_r(Y)$ is open in $Y$ by Corollary \ref{cor:regular1} and $x$ is a point in the interior of $X$ in $\myReg_r(Y)$.
	By Lemma \ref{lem:regular}, there exist a coordinate projection $\pi:F^n \to F^d$ and a definable $\mathcal C^r$ map $\varphi: \pi(U) \to F^{n-d}$ such that $U \cap Y$ is a permutated graph of $\varphi$ under $\pi$ by shrinking $U$ if necessary because $x$ is $r$-regular in $Y$. 
\end{proof}

\begin{lemma}\label{lem:sum_large}
	Consider a d-minimal expansion of an ordered field.
	Let $Y$ be a definable set and $X_i$ be a definable subset of $Y$ for $1 \leq i \leq m$.
	If the union $\bigcup_{i=1}^m X_i$ is not lean in $Y$, the definable set $X_i$ is not lean in $Y$ for some $1 \leq i \leq m$.
	In addition, the union of finitely many definable subsets which are lean in $Y$ is lean in $Y$, and the intersection of finitely many definable subsets which are huge in $Y$ is huge in $Y$.
\end{lemma}
\begin{proof}
	Since $\bigcup_{i=1}^m X_i$ is not lean in $Y$, we have 
	\begin{align*}
	\dim \left(\bigcup_{i=1}^m \myint_Y(X_i)\right) = \dim \left(\myint_Y\left(\bigcup_{i=1}^m X_i \right)\right) =\dim Y.
	\end{align*}
	We have $\dim \myint_Y(X_i)=\dim Y$ for some $1 \leq i \leq m$ by Proposition \ref{prop:dim}(2).
	It means that $X_i$ is not lean.
	The `in addition' part is obvious.
\end{proof}

\begin{lemma}\label{lem:huge_map}
	Let $r$ be a positive integer.
	Consider a d-minimal expansion of an ordered field.
	Let $X$ be a nonempty definable set and $f:X \to F^n$ be a definable map, where $F$ is the universe of the structure.
	Then, there exists a definable subset $U$ of $X$ such that $U$ is open and huge in $X$, $U$ is a definable $\mathcal C^r$ submanifold and the restriction of $f$ to $U$ is of class $\mathcal C^r$.
\end{lemma}
\begin{proof}
	Note that $\myReg_r(X)$ is open in $X$ and it is a definable $\mathcal C^r$ submanifold by Corollary \ref{cor:regular1} and Corollary \ref{cor:regular2}.
	There exists a definable open subset $U$ of $\myReg_r(X)$ such that $\dim \myReg_r(X) \setminus U<\dim X$ and the restriction of $f$ to $U$ is of class $\mathcal C^r$ by Lemma \ref{lem:mfd_cr}.
	The definable set $U$ is a definable $\mathcal C^r$ submanifold by Lemma \ref{lem:open_mfd}.
	The definable set $U$ is obviously open in $X$ and huge in $X$ by Proposition \ref{prop:lean_equiv}.
\end{proof}

\begin{lemma}\label{lem:lean_imply_ineq}
Consider a d-minimal expansion of an ordered field.
Let $X$ and $Y$ be definable sets with $X \subseteq Y$.
If $\dim X=\dim Y$ and $X$ is lean in $Y$, then the inequality $\mypdeg X < \mypdeg Y$ holds.
\end{lemma}
\begin{proof}
	Let $r$ be a positive integer.
	By considering the contraposition of the lemma together with Proposition \ref{prop:lean_equiv}, we have only to show that $\dim(X \cap \myReg_r(Y))=\dim Y$ if the equalities $\dim X=\dim Y$ and $\mypdeg(X)=\mypdeg(Y)$ holds.
	Set $d=\dim Y$ and $m=\mypdeg(Y)$.
	Assume for contradiction that $\dim(X \cap \myReg_r(Y))<d$.
	We have $\dim X \setminus (Y \setminus \myReg_r(Y))= \dim(X \cap \myReg_r(Y))<d$.
	Therefore, the equality $\mypdeg(Y)=\mypdeg(X) =\mypdeg (X \cap (Y \setminus \myReg_r(Y)))\leq \mypdeg(Y \setminus \myReg_r(Y))$ holds by  Lemma \ref{lem:small_gap} and Lemma \ref{lem:partition_included}.
	It is obvious that $\mypdeg(Y \setminus \myReg_r(Y))=m-1$ by Proposition \ref{prop:regular_seq}, which is impossible.
\end{proof}

\begin{lemma}\label{lem:lean_product}
	Consider a d-minimal expansion of an ordered field.
	Let $X$ and $Y$ be definable sets.
	Let $\pi_1:X \times Y \to X$ and $\pi_2:X \times Y \to Y$ be the projections.
	Let $Z$ be a definable subset of $X \times Y$ and set $W:=\{y \in Y\;|\; \pi_1(\pi_2^{-1}(y) \cap Z) \text{ is not lean in }X\}.$
	If $Z$ is lean in $X \times Y$, then $W$ is lean in $Y$.
\end{lemma}
\begin{proof}
	We prove the contraposition.
	Fix a positive integer $r$.
	Assume that $W$ is not lean in $Y$.
	We have $\dim (W \cap \myReg_r(Y))=\dim Y$ by Proposition \ref{prop:lean_equiv}.
	The same proposition asserts the equality  $W=\{y \in Y\;|\; \dim (\pi_1(\pi_2^{-1}(y) \cap Z) \cap \myReg_r(X))=\dim X\}$.
	Apply Corollary \ref{cor:dim1}(2) to the definable set $\pi_2^{-1}(W) \cap (\myReg_r(X) \times \myReg_r(Y)) \cap Z$ and the restriction of $\pi_2$ to this set.
	We have $\dim (\pi_2^{-1}(W) \cap (\myReg_r(X) \times \myReg_r(Y)) \cap Z) = \dim X + \dim Y$.
	We get $\dim (\myReg_r(X \times Y) \cap Z)=\dim ((\myReg_r(X) \times \myReg_r(Y)) \cap Z) =\dim X+\dim Y$ by Corollary \ref{cor:product_reg}.
	It means that $W$ is not lean in $X \times Y$ by Proposition \ref{prop:lean_equiv}.
\end{proof}

We can generalize Lemma \ref{lem:lean_product} as follows:
We do not use the following proposition in this paper.
\begin{proposition}\label{prop:lean_product2}
	Consider a d-minimal expansion of an ordered field $\mathcal F=(F,<,+,\cdot,0,1,\ldots)$.
	Let $\pi:F^n \to F^m$ be a coordinate projection.
	Let $X$ and $Z$ be definable subsets of $F^n$ such that $Z \subseteq X$ and $\dim X=\dim Z$.
	Assume that $\dim (X \cap \pi^{-1}(x))$ is independent of $x \in \pi(X)$.
	Set $W:=\{x \in \pi(X)\;|\; \pi^{-1}(x) \cap Z \text{ is not lean in }\pi^{-1}(x) \cap X\}.$
	If $Z$ is lean in $X$, then $W$ is lean in $\pi(X)$.
\end{proposition}
\begin{proof}
	Fix a positive integer $r$.
	We may assume that $\pi$ is the projection onto the first $m$ coordinates by permuting the coordinate if necessary.
	Set $d=\dim \pi(X)$ and $e=\dim (X \cap \pi^{-1}(x))$ for $x \in \pi(X)$.
	The number $e$ is independent of choice of $x \in \pi(X)$ by the assumption.
	We prove the contraposition.
	Assume that $W$ is not lean in $\pi(X)$.
	The definable set $W$ is $r$-large in $\pi(X)$ by Proposition \ref{prop:lean_equiv}.
	There exist an open box $U$ in $F^m$, a coordinate projection $\pi_1:F^m \to F^d$ and a definable $\mathcal C^r$ map $\varphi: \pi_1(U) \to F^{m-d}$ such that $U \cap W = U \cap \pi(X)$ and $U \cap \pi(X)$ is a permutated graph of $\varphi$ under $\pi_1$ by Lemma \ref{lem:large_equiv}. 
	We may assume that $\pi_1$ is the coordinate projection onto the first $d$ coordinates as usual.
	Let $\widetilde{\varphi}:\pi_1(U) \to F^m$ be the definable map given by $\widetilde{\varphi}(x)=(x,\varphi(x))$.
	Note that $U \cap W=U \cap \pi(X)$ is contained in $\myReg_r(\pi(X))$ by Lemma \ref{lem:regular}.
	Therefore, $U \cap W$ is a definable $\mathcal C^r$ submanifold of $F^m$ of dimension $d$ by Corollary \ref{cor:regular1} and Corollary \ref{cor:regular2}.
	
	Let $\Pi$ be the set of coordinate projections of $F^{n-m}$ onto $F^e$.
	Set $X_x:=\{y \in F^{n-m}\;|\; (x,y) \in X\}$ for each $x \in F^m$.
	We define $Z_x$ in the same manner.
	For any $\rho \in \Pi$, we define the formula $\Phi_{\rho}(x,y,z)$ as follows:
	\begin{align*}
		\Phi_\rho(x,y,z) =&\left( \bigwedge_{i=1}^{n-m} y_i<z_i \right)\wedge \left(Z_x \cap U_{y,z}=X_x \cap U_{y,z}\right) \\
		&\wedge X_x \cap U_{y,z} \text{ is the permutated graph of a definable }\mathcal C^r\text{ map}\\
		&\quad \text{ defined on }\rho(U_{y,z}) \text{ under }\rho.
	\end{align*}
	Here, $y=(y_1,\ldots, y_{n-m})$, $z=(z_1,\ldots, z_{n-m})$ and  $U_{y,z}$ is the open box given by $\prod_{i=1}^{n-m} (y_i,z_i)$.
	
	By Proposition \ref{prop:lean_equiv} and Lemma \ref{lem:large_equiv}, for any $x \in W$,  we have $\mathcal F \models \exists y\ \exists z\ \Phi_\rho(x,y,z)$ for some $\rho \in \Pi$.
	Set $V_{\rho}=\{x \in U \cap W\;|\; \mathcal F \models \exists y\ \exists z\ \Phi_\rho(x,y,z)\}$, then we have $U \cap W = \bigcup_{\rho \in \Pi}V_{\rho}$.
	We have $\dim V_{\rho_1} = \dim U \cap W=d$ for some $\rho_1 \in \Pi$ by Proposition \ref{prop:dim}(2) because $\Pi$ is a finite set.
	It implies that $V_{\rho_1}$ has a nonempty interior in $\myReg_r(\pi(X))$ by Corollary \ref{cor:regular2} and Corollary \ref{cor:mfd_same_dim}.
	Therefore, we may assume that $V_{\rho_1}$ coincides with $U \cap W$ by shrinking $U$ if necessary.
	We may assume that $\rho_1$ is the projection onto the first $e$ coordinates by permuting the coordinates if necessary.
	By Lemma \ref{lem:definable_selection}, there exist definable maps $\tau_1,\tau_2:U \cap W \to F^{n-m}$ such that $\mathcal F \models \Phi_{\rho_1}(x,\tau_1(x),\tau_2(x))$ for each $x \in U \cap W$. 
	We may assume that $\tau_1$ and $\tau_2$ are of class $\mathcal C^r$ by Lemma \ref{lem:mfd_cr} by shrinking $U$ if necessary.
	Set $D=\{(x,y) \in \pi_1(U) \times F^e\;|\; y \in \rho_1(U_{\tau_1(\widetilde{\varphi}(x)), \tau_2(\widetilde{\varphi}(x))})\})$.
	It is obviously an open subset of $F^{d+e}$.
	Recall that $X_{\widetilde{\varphi}(x)} \cap U_{\tau_1(\widetilde{\varphi}(x)), \tau_2(\widetilde{\varphi}(x))}$ is the graph of a definable $\mathcal C^r$ map defined on $ \rho_1(U_{\tau_1(\widetilde{\varphi}(x)), \tau_2(\widetilde{\varphi}(x))})$, say $f_x$.
	We consider the definable map $\psi:D \to F^{m-d} \times F^{n-m-e}$ given by $\psi(x,y)=(\varphi(x),f_x(y))$.
	There exist a nonempty open box $C_1$ contained in $\pi_1(U)$ and a nonempty open box $C_2$ in $F^e$ such that $C_1 \times C_2 \subseteq D$ and the restriction of $\psi$ to $C_1 \times C_2$ is of class $\mathcal C^r$ by Lemma \ref{lem:dim3}.
	Take a point $(s,t) \in C_1 \times C_2$.
	We next take a small positive $\varepsilon>0$ so that $\mathcal B_m(\widetilde{\varphi}(s),\varepsilon):=\{x \in F^m\;|\; |x-\widetilde{\varphi}(s)|<\varepsilon\} \subseteq U$ and $|p_i(f_s(t))-p_i(\tau_j(\widetilde{\varphi}(s)))|>\varepsilon$ for each $1 \leq i \leq n-m-e$ and $j=1,2$, where $p_i:F^{n-m-e} \to F$ denotes the coordinate projection onto the $i$-th coordinate.
	Take $\delta>0$ so that $\delta<\varepsilon$, $(s',t') \in C_1 \times C_2$, $ |\psi(s',t')-\psi(s,t)|<\varepsilon/2$ and $|\tau_j(s')-\tau_j(s)|<\varepsilon/2$ for $j=1,2$ whenever $|s-s'|<\delta$ and $|t-t'|<\delta$. 
	Set $O=(U \cap (\mathcal B_d(s,\delta) \times F^{m-d})) \times \mathcal B_e(t,\delta) \times \mathcal B_{n-m-e}(f_s(t),\varepsilon/2)$.
	We have $O \cap Z=O \cap X$ and $O \cap X$ is the permutated graph of the restriction of the definable $\mathcal C^r$ map $\psi$.
	It implies that $Z$ is $r$-large in $X$ by Lemma \ref{lem:large_equiv}, which is equivalent to that $Z$ is not lean by Proposition \ref{prop:lean_equiv}.
\end{proof}

\section{Definable $\mathcal C^r$ structures on definable topological groups}\label{sec:defnable_grpup}

We prove the main theorem in this section.
We first recall several definitions.
\begin{definition}\label{def:topology}
	Let $\mathcal M=(M,\ldots)$ be a model-theoretic structure.
	Let $X$ and $T$ be definable sets.
	A parameterized family $\{S_t\;|\;t \in T\}$ of definable subsets of $X$ is called \textit{definable} if the union $\bigcup_{t \in T} \{t\} \times S_t$ is definable.
	
	A topological space $(X,\tau)$ is a \textit{definable topological space} if $X$ is a definable set and there exists a definable parametrized family of sets $\mathcal B$ which is an open base for the topology $\tau$.
	We call the open base $\mathcal B$ a \textit{definable open base} of the topology $\tau$.
	%
	
	Consider the case in which $\mathcal M$ is an expansion of a dense linear order $(M,<)$.
	The linear order $<$ induces a topology on $M$ called the \textit{order topology}.
	The Cartesian product $M^n$ equips the product topology of the order topology.
	Any definable subset $X$ of $M^n$ has the relative topology induced from the product topology of $M^n$.
	It is called the \textit{affine topology} on $X$ in this paper.
\end{definition}

\begin{definition}
	Consider a model-theoretic structure $\mathcal M$.
	A group $(G,\cdot)$ is \textit{definable} if both the underlying space $G$ and the multiplication $\cdot:G \times G \to G$ are definable. 
	
	We call the pair of a definable group $G$ and a definable topology $\tau$ on $G$ is 
	a \textit{definable topological group} if the multiplication and the inverse in $G$ are continuous maps under the topology $\tau$.
	We consider the affine topology in this paper unless the topology is explicitly given.
\end{definition}

\begin{definition}\label{def:mfd}
	Consider a definably complete expansion of an ordered field.
	Let $r$ be a positive integer.
	An \textit{abstract definable $\mathcal C^r$ manifold} $((X,\tau),(\varphi_i:U_i \to V_i)_{i \in I})$  is a pair of a definable topological space $(X,\tau)$ and a finite family of definable homeomorphisms $(\varphi_i:U_i \to V_i)_{i \in I}$ from an open set $U_i$ of $X$ onto a definable $\mathcal C^r$ submanifold $V_i$ such that $(U_i)_{i \in I}$ is a finite cover of $X$ by definable open subsets of $X$ such that the induced map $\varphi_{ij}:\varphi_i(U_i \cap U_j) \ni x \mapsto \varphi_j(\varphi_i^{-1}(x)) \in \varphi_j(U_i \cap U_j)$ is a definable $\mathcal C^r$ diffeomorphism for each $i,j \in I$ whenever $U_i \cap U_j \neq \emptyset$. 
	We call the family $(\varphi_i:U_i \to V_i)_{i \in I}$ a \textit{definable $\mathcal C^r$ atlas} of the abstract definable $\mathcal C^r$ manifold.
	The Cartesian product of two abstract definable $\mathcal C^r$ manifolds is naturally an abstract definable $\mathcal C^r$ manifold.
	We often write $(X,(\varphi_i:U_i \to V_i)_{i \in I})$ instead of $((X,\tau),(\varphi_i:U_i \to V_i)_{i \in I})$  because the topology $\tau$ is uniquely determined by the family of definable bijective maps $(\varphi_i:U_i \to V_i)_{i \in I}$.
	
	A \textit{definable map} $f:(X,(\varphi_i:U_i \to V_i)_{i \in I}) \to (Y,(\psi_j:U'_j \to V'_j)_{j \in J})$ between abstract definable $\mathcal C^r$ manifolds is a definable map from the definable set $X$ to the definable set $Y$.
	We simply denote it by $f:X \to Y$ when the definable $\mathcal C^r$ atlases are clear from the context.
	A definable map $f:(X,(\varphi_i:U_i \to V_i)_{i \in I}) \to (Y,(\psi_j:U'_j \to V'_j)_{j \in J})$ is \textit{of class $\mathcal C^r$} if the composition $\psi_j \circ f \circ \varphi_i^{-1}: \varphi_i^{-1}(f^{-1}(U'_j) \cap U_i) \to V'_j$ is of class $\mathcal C^r$ for each $i \in I$ and $j \in J$ whenever $f^{-1}(U'_j) \cap U_i \neq \emptyset$.
\end{definition}

\begin{definition}\label{def:crstr}
	Consider a definably complete expansion of an ordered field.
	Let $r$ be a positive integer.
	Let $G$ be a definable topological group.
	A \textit{definable $\mathcal C^r$ structure on $G$} is a finite family  $(\varphi_i:U_i \to V_i)_{i \in I}$ of definable homeomorphisms such that $(G, (\varphi_i:U_i \to V_i)_{i \in I})$ is an abstract definable $\mathcal C^r$ manifold and the multiplication and inversion are of class $\mathcal C^r$ as definable maps between abstract definable $\mathcal C^r$ manifolds. 
\end{definition}

We are now ready to show that a definable topological group is covered by a finitely many translates of a huge definable subset.
It is proven in Proposition \ref{porp:top_grp_cover}.
We have introduced the notion of partition degree to prove it by induction.
We first prove a lemma.

\begin{lemma}\label{lem:top_grp_cover_sub}
	Consider a d-minimal expansion of an ordered field.
	Let $G$ be a definable topological group and $V$ be a definable subset of $G$ which is huge in $G$.
	Let $W$ be a definable subset of $G$.
	Set $A=\{g \in G\;|\; W \cap (G \setminus gV) \text{ is lean in }W\}$.
	The definable set $A$ is not lean in $G$.
	In particular, $A$ is not an empty set.
\end{lemma}
\begin{proof}
	Set $B=\{g \in G\;|\; W \cap (G \setminus gV) \text{ is not lean in }W\}$.
	We have $B=G \setminus A$.
	Assume for contradiction that $A$ is lean in $G$.
	The definable set $B$ is not lean in $G$ by Lemma \ref{lem:sum_large}.
	Consider the definable set $X=\{(g,w) \in G \times W\;|\; w \in W \cap (G \setminus gV)\}$.
	The definable set $X$ is not lean in $G \times W$ by Lemma \ref{lem:lean_product}.
	Fix a positive integer $r$.
	We have $\dim (X \cap \myReg_r(G \times W))=\dim (G \times W)$ by Proposition \ref{prop:lean_equiv}.
	We get $\dim (X \cap (\myReg_r(G) \times \myReg_r(W))) = \dim (G \times W)$ by Corollary \ref{cor:product_reg}.
	
	For any $w \in W$, we set $X^w:=\{g \in G\;|\; (g,w) \in X\} \subseteq \{g \in G\;|\; w \notin gV\}=w \cdot  (G \setminus V)^{-1}$.
	The definable set $w \cdot  (G \setminus V)^{-1}$ is lean in $G$ by Lemma \ref{lem:lean_homeo} because the inversion and multiplication by $w$ are definable homeomorphisms.
	The definable subset $X^w$  of $w \cdot  (G \setminus V)^{-1}$ is also lean in $G$.
	This implies the inequality $\dim (X^w \cap \myReg_r(G))<\dim G$ for each $w \in W$ by Proposition \ref{prop:lean_equiv}.
	We get $\dim (X \cap (\myReg_r(G) \times \myReg_r(W))) < \dim G + \dim W = \dim (G \times W)$ by Corollary \ref{cor:dim1}(2), which is absurd.
\end{proof}

\begin{proposition}\label{porp:top_grp_cover}
	Consider a d-minimal expansion of an ordered field.
	Let $G$ be a definable topological group and $V$ be a definable subset of $G$ which is huge in $G$.
	There exist finitely many elements $g_1,\ldots, g_m$ in $G$ such that $G=\bigcup_{i=1}^m g_iV$.
\end{proposition}
\begin{proof}
	We set $W_0=G$ and $W_i=G \setminus (\bigcup_{j=1}^i g_jV)$ for $i>0$ when elements $g_1,\ldots, g_i  \in G$ are fixed.
	Note that any strictly decreasing sequence of $\mathbb Z_{\geq 0} \times \mathbb Z_{\geq 0} $ under the lexicographic order is of finite length, where $\mathbb Z_{\geq 0}$ denotes the set $\{z \in \mathbb Z\;|\; z \geq 0\}$.
	Therefore, if we can choose $g_{i+1} \in G$ so that at least one of the inequalities $\dim W_{i+1}<\dim W_i$ and $\mypdeg W_{i+1}<\mypdeg W_i$ holds, we get $W_n=\emptyset$ for some positive integer $n$.
	Apply Lemma \ref{lem:top_grp_cover_sub}, then we can choose $g_{i+1} \in G$ so that $W_{i+1}= G \setminus (\bigcup_{j=1}^{i+1} g_jV)=W_i \cap (G \setminus g_{i+1}V) $ is lean in $W_i$.
	We have $\dim W_{i+1}<\dim W_i$ or $\mypdeg W_{i+1}<\mypdeg W_i$ by Lemma \ref{lem:lean_imply_ineq}.
\end{proof}

We next prove several lemmas used in the proof of the main theorem.

\begin{lemma}\label{lem:basic_grp1}
	Consider a d-minimal expansion of an ordered field.
	Let $G$ be a definable topological group and $V$ be a definable subset of $G$ which is huge in $G$.
	For any $g \in G$, there exist $g_1,g_2 \in V$ such that $g=g_1g_2$.
\end{lemma}
\begin{proof}
	Consider the definable map $f:G \to G$ given by $f(h)=gh^{-1}$.
	It is a definable homeomorphism because $G$ is a definable topological group.
	The image $f(V)$ is huge by Lemma \ref{lem:lean_homeo}.
	The intersection $f(V) \cap V$ is huge by Lemma \ref{lem:sum_large}.
	In particular, the intersection is not empty.
	There exists $g_1, g_2 \in V$ such that $g_1=f(g_2)$.
	This implies that $g=g_1g_2$.
\end{proof}

\begin{lemma}\label{lem:basic_grp2}
	Consider a d-minimal expansion of an ordered field.
	Let $G$ be a definable topological group and $V$ be a definable subset of $G$ which is huge in $G$.
	Then the set $\{(g_1,g_2) \in G \times G\;|\; g_1g_2 \in V\}$ is huge in $G \times G$
\end{lemma}
\begin{proof}
	Set $X=\{(g_1,g_2) \in G \times G\;|\; g_1g_2 \notin V\}$.
	We have only to prove that $X$ is lean in $G \times G$.
	Consider the definable homeomorphism $f:G \times G \to G \times G$ given by $f(g_1,g_2)=(g_1,g_1g_2)$.
	We have $f(X)=G \times (G \setminus V)$.
	Fix a positive integer $r$.
	The inequality $\dim (\myReg_r(G) \setminus V)<\dim G$ holds by Proposition \ref{prop:lean_equiv}.
	We have $\dim ((G \times (G \setminus V)) \cap \myReg_r(G \times G)) = \dim ((G \times (G \setminus V)) \cap (\myReg_r(G)\times \myReg_r(G)))=\dim (\myReg_r(G) \times (\myReg_r(G) \setminus V))<2\dim G=\dim(G \times G)$ by  Corollary \ref{cor:product_reg} and Corollary \ref{cor:dim1}(2).
	It implies that $G \times (G \setminus V)$ is lean in $G \times G$ by Proposition \ref{prop:lean_equiv}.
	Since $f$ is a definable homeomorphism, $X$ is lean in $G \times G$ by Lemma \ref{lem:lean_homeo}.
\end{proof}

\begin{lemma}\label{lem:basic_grp3}
	Consider a d-minimal expansion of an ordered field.
	Let $G$ be a definable topological group and $V$ be a definable subset of $G$ which is huge in $G$.
	Let $Y$ be a definable subset of $G \times G$ which is huge in $G \times G$.
	The definable subsets 
	\begin{align*}
		&W_1=\{a \in V\;|\; \{b \in G\;|\; (b,a) \in Y\} \text{ is huge in }G\} \text{ and }\\
		&W_2=\{a \in V\;|\; \{b \in G\;|\; (b^{-1},ba) \in Y\} \text{ is huge in }G\}
	\end{align*}
	of $G$ are huge in $G$.
\end{lemma}
\begin{proof}
	Fix a positive integer $r$.
	We first prove that $W_1$ is huge in $G$.
	Assume for contradiction that $W_1$ is not huge.
	The definable set $G \setminus W_1$ is not lean.
	The equality $\dim (\myReg_r(G) \setminus W_1)=\dim G$ holds by Proposition \ref{prop:lean_equiv} .
	We want to show $\dim ((\myReg_r(G) \cap V) \setminus W_1)=\dim G$.
	Since the inclusion $\myReg_r(G) \setminus W_1 \subseteq ((\myReg_r(G) \cap V) \setminus W_1) \cup (\myReg_r(G) \setminus V)$ holds, we have $\dim G=\dim (\myReg_r(G) \setminus W_1) \leq \max\{\dim ((\myReg_r(G) \cap V) \setminus W_1), \dim (\myReg_r(G) \setminus V)\} \leq \dim G$ by Proposition \ref{prop:dim}(2).
	On the other hand, we get $\dim  (\myReg_r(G) \setminus V) < \dim G$ by Proposition \ref{prop:lean_equiv} because $V$ is huge in $G$.
	Consequently, we get $\dim ((\myReg_r(G) \cap V) \setminus W_1)=\dim G$.
	
	For each $a \in V$, set $Y_a=\{b \in G\;|\; (b,a) \in Y\}$.
	For any $a \in (\myReg_r(G) \cap V) \setminus W_1$, the equality $\dim( \myReg_r(G) \setminus Y_a)=\dim G$ holds by Proposition \ref{prop:lean_equiv} because $Y_a$ is not huge in $G$.
	Consider the definable set $$Z=\bigcup_{a \in (\myReg_r(G) \cap V) \setminus W_1}((\myReg_r(G) \setminus Y_a) \times \{a\}).$$
	We have $\dim Z=2\dim G=\dim(G \times G)$ by Corollary \ref{cor:dim1}(2).
	On the other hand, the inclusion $Z \subseteq \bigcup_{a \in \myReg_r(G)}((\myReg_r(G) \setminus Y_a) \times \{a\}) =\myReg_r(G )\times \myReg_r(G ) \setminus Y=\myReg_r(G \times G) \setminus Y$ holds by Corollary \ref{cor:product_reg}.
	We have $\dim (\myReg_r(G \times G) \setminus Y) \geq \dim Z=2\dim G$, which contradicts the assumption that $Y$ is huge in $G \times G$ by Proposition \ref{prop:lean_equiv}.
	
	We next prove that $W_2$ is huge in $G$.
	Consider the definable map $\rho:G \times G \to G \times G$ given by $\rho(b,a)=(b^{-1},ba)$.
	Since $G$ is a definable topological group, the map $\rho$ is a homeomorphism.
	The definable set $\rho^{-1}(Y)$ is huge in $G \times G$ by Lemma \ref{lem:lean_homeo}.
	We have $W_2=\{a \in V\;|\; \{b \in G\;|\; (b,a) \in \rho^{-1}(Y)\} \text{ is huge in }G\}$.
	The definable set $W_2$ is huge in $G$ by the first part of the lemma which we have already shown.
\end{proof}

We are now ready to prove the main theorem.

\begin{theorem}\label{thm:cr_structure_on_group}
Consider a d-minimal expansion of an ordered field.
Let $r$ be a positive integer.
Let $G$ be a definable topological group.
There exists a definable open subset $V$ and finitely many elements $g_1,\ldots, g_m$ of $G$ such that $V$ is a definable $\mathcal C^r$ submanifold and $(\varphi_i:U_i:=g_iV \ni g_i \cdot g \mapsto g \in V)_{1 \leq i \leq m}$ is a definable $\mathcal C^r$ structure on $G$.
\end{theorem}
\begin{proof}
	The proof of this theorem is very similar to that of \cite[Theorem 3.5]{Wencel}.
	However, we give a complete proof here for readers' convenience.
	
	We first construct a definable subset $V$ of $G$, a definable subset $Y$ of $G \times G$ and finitely many elements $g_1,\ldots g_m$ in $G$ satisfying the following conditions:
	\begin{enumerate}
		\item[(a)] the set $V$ is open and huge in $G$ and it is a definable $\mathcal C^r$ submanifold;
		\item[(b)] the inversion is a $\mathcal C^r$ map from $V$ onto $V$;
		\item[(c)] the set $Y$ is open and huge in $G \times G$, and the restriction of the multiplication to $Y$ is of class $\mathcal C^r$ and assumes values in $V$;
		\item[(d)] for every $a \in V$ and $a' \in G$, the definable sets $U_1(a):=\{b \in G\;|\; (b,a) \in Y\}$, $U_2(a):=\{b \in G\;|\; (b^{-1},ba) \in Y\}$ and $U_3(a,a'):=\{b \in G\;|\;(ba',a) \in Y\}$ are huge in $G$;
		\item[(e)] $G=\bigcup_{i=1}^m g_iG$.
	\end{enumerate}

We can take a definable subset $V_0$ of $G$ such that $V_0$ is open and huge in $G$, it is a definable $\mathcal C^r$ submanifold and the restriction of the inversion to $V_0$ is of class $\mathcal C^r$ by Lemma \ref{lem:huge_map}.
Let $f:G \times G \to G$ be the multiplication in $G$.
The inverse image $f^{-1}(V_0)$ is open and huge in $G \times G$ by Lemma \ref{lem:basic_grp2}.
Therefore,  $(V_0 \times V_0) \cap  f^{-1}(V_0)$ is open and huge in $G \times G$ by Lemma \ref{lean:product} and Lemma \ref{lem:sum_large}.
We can take a definable subset $Y_0$ of $G \times G$ such that $Y_0$ is open and huge in $G \times G$ and the restriction of $f$ to $Y_0$ is of class $\mathcal C^r$ by Lemma \ref{lem:huge_map}.
Set $Y_1:= Y_0 \cap (V_0 \times V_0) \cap  f^{-1}(V_0)$.
It is open and huge in $G \times G$ by Lemma \ref{lem:sum_large}.
It is obvious the the restriction of $f$ to $Y_1$ is of class $\mathcal C^r$.

Now we define 
	\begin{align*}
	&W_1=\{a \in V_0\;|\; \{b \in G\;|\; (b,a) \in Y_1\} \text{ is huge in }G\};\\
	&W_2=\{a \in V_0\;|\; \{b \in G\;|\; (b^{-1},ba) \in Y_1\} \text{ is huge in }G\}.
\end{align*}
They are huge in $G$ by Lemma \ref{lem:basic_grp3}.
Set $V_1=W_1 \cap W_2$ which is huge in $G$ by Lemma \ref{lem:sum_large}.
Set $V_2=\myint_G(V_1)$ and $V=V_2 \cap V_2^{-1}$.
The definable set $V_2$ is open and huge in $G$ by Lemma \ref{lem:lean_interior} and $V$ is open and huge in $G$ by Lemma \ref{lem:lean_homeo} and Lemma \ref{lem:sum_large}.
The restriction of inversion to $V$ is a definable $\mathcal C^r$ map from $V$ to $V$.
The definable set $V$ is a definable $\mathcal C^r$ submanifold by Lemma \ref{lem:open_mfd} because $V$ is an open subset of $V_0$.
The definable set $V$ fulfills conditions (a) and (b).

Set $Y=(V \times V) \cap Y_1 \cap f^{-1}(V)$. 
It is obvious that $Y$ is open in $G$ and $f(Y) \subseteq V$.
The inclusion $Y \subseteq Y_0$ implies that the restriction of $f$ to $Y$ is of class $\mathcal C^r$.
Note that $f^{-1}(V)$ is huge in $G$ by Lemma \ref{lem:basic_grp2}.
The definable set $Y$ is huge in $G$ by Lemma \ref{lean:product} and Lemma \ref{lem:sum_large}.
We have shown that $Y$ fulfills condition (c).

The following equalities hold:
\begin{align*}
	&U_1(a)= \{b \in G\;|\; (b,a) \in Y\}= V \cap Va^{-1} \cap \{b \in G\;|\; (b,a) \in Y_1\};\\
	&U_2(a)= \{b \in G\;|\; (b^{-1},ba) \in Y\}= V \cap Va^{-1} \cap \{b \in G\;|\; (b^{-1},ba) \in Y_1\};\\
	&U_3(a,a')=\{b \in G\;|\; (ba',a) \in Y\}=U_1(a) \cdot (a')^{-1}.
\end{align*}
Since $V \subseteq W_1 \cap W_2$, by the definition of $W_1$ and $W_2$, the definable sets $\{b \in G\;|\; (b,a) \in Y_1\}$ and $\{b \in G\;|\; (b^{-1},ba) \in Y_1\}$ are huge in $G$ for each $a \in V$.
Using this fact, Lemma \ref{lem:sum_large} and Lemma \ref{lem:lean_homeo}, we can prove that $U_1(a)$, $U_2(a)$ and $U_3(a,a')$ are huge in $G$ .
The proof is left to readers.
We have demonstrated that condition (d) is fulfilled by $V$ and $Y$.

Finally, because $V$ is huge in $G$, we can take $g_1,\ldots, g_m \in G$ satisfying the equality $G=\bigcup_{i=1}^m g_iV$ by Proposition \ref{porp:top_grp_cover}.
We have constructed definable sets $V$, $Y$ and finitely many elements $g_1,\ldots,g_m \in G$ satisfying conditions (a) through (e).
\medskip

We next prove two claims.
\medskip

\textbf{Claim 1.}
For any $a,b \in G$, the set $Z(a,b):=\{x \in V\;|\;a \cdot x \cdot b \in V\}$ is open in $V$ and the map $g_{a,b}: Z(a,b) \ni x \mapsto a \cdot x \cdot b \in a \cdot Z(a,b) \cdot b$ is a definable $\mathcal C^r$ diffeomorphism.
\begin{proof}[Proof of Claim 1]
	It is obvious that $Z(a^{-1},b^{-1})=a \cdot Z(a,b) \cdot b$ and $g_{a^{-1},b^{-1}}$ is the inverse of $g_{a,b}$.
	Therefore, we have only to prove that $g_{a,b}$ is a definable $C^r$ map for each $a,b \in G$.
	We have only to show that, for each $a,b \in G$ and $x_0 \in Z(a,b)$, there exists a definable open subset $U$ of $G$ such that $x_0 \in U \subseteq Z(a,b)$ and the restriction of $g_{a,b}$ is of class $\mathcal C^r$.
	
	Fix arbitrary elements $a,b \in G$ and $x_0 \in Z(a,b)$.
	There exist $b_1,b_2 \in V$ such that $b=b_1 \cdot b_2$ by Lemma \ref{lem:basic_grp1}.
	Consider the definable set 
	\begin{align*}
		Z_0&:=\{c \in V\;|\; (c \cdot a, x_0 ) \in Y,(c \cdot a \cdot x_0,b_1) \in Y, (c \cdot a \cdot x_0 \cdot b_1, b_2) \in Y, \\
		&\quad (c^{-1}, c \cdot a \cdot x_0 \cdot b) \in Y\}\\
		&=U_3(x_0,a) \cap U_3(b_1,a \cdot x_0) \cap U_3(b_2,a \cdot x_0 \cdot b_1) \cap U_2(a \cdot x_0 \cdot b).
	\end{align*}
	The definable set $Z_0$ is huge in $G$ by condition (d) and Lemma \ref{lem:sum_large}.
	In particular, the definable set $Z_0$ is not empty. 
	We take a point $c \in Z_0$ and set
	\begin{align*}
		U&:=\{x \in V\;|\; (c \cdot a, x) \in Y,(c \cdot a \cdot x,b_1) \in Y, (c \cdot a \cdot x \cdot b_1, b_2) \in Y, \\
		&\quad (c^{-1}, c \cdot a \cdot x \cdot b) \in Y\}.
	\end{align*}
	The definable set $U$ is nonempty and open in $G$ because $x_0 \in U$ and, $V$ and $Y$ are open in $G$ and $G \times G$, respectively. 
	Recall that the restriction of multiplication to $Y$ is of class $C^r$ and assumes values in $V$.
	Note that the map $g_{a,b}$ is the composition of definable $\mathcal C^r$ maps:
	\begin{align*}
		x &\mapsto (c \cdot a, x) \mapsto c \cdot a \cdot x \mapsto (c \cdot a \cdot x, b_1) \mapsto  c \cdot a \cdot x \cdot b_1\\
		&\mapsto (c \cdot a \cdot x \cdot b_1, b_2) \mapsto c \cdot a \cdot x \cdot b \mapsto (c^{-1}, c \cdot a \cdot x\cdot b) \mapsto a \cdot x \cdot b.
	\end{align*}
	Every elements in $G$ appeared in the above formula such as $c \cdot a \cdot x $ belongs to $V$, and every elements in $G \times G$ appeared in the above formula such as $(c \cdot a,x)$ belongs to $Y$ whenever $x \in U$.
	Consequently, the restriction of $g_{a,b}$ to $U$ is of class $\mathcal C^r$ by condition (c).
	We have completed the proof of Claim 1.
\end{proof}

\textbf{Claim 2.}
For any $a,b \in G$, the set $Z'(a,b):=\{(x,y) \in V \times V\;|\;a \cdot x \cdot b \cdot y \in V\}$ is open in $V \times V$ and the map $h_{a,b}: Z'(a,b) \ni (x,y) \mapsto a \cdot x \cdot b \cdot y \in V$ is a definable $\mathcal C^r$ map.
\begin{proof}[Proof of Claim 2]
Fix arbitrary $a, b \in G$ and $(x_0,y_0) \in Z'(a,b)$.
In the same manner as Claim 1, we have only to show that there exists a definable open subset $U$ of $Z'(a,b)$ containing the point $(x_0,y_0)$ to which the restriction of $h_{a,b}$ is of class $\mathcal C^r$.

We can take $b_1,b_2 \in V$ so that $b=b_1 \cdot b_2$ by Lemma \ref{lem:basic_grp1}.
	Consider the definable set 
\begin{align*}
	Z'_0&:=\{c \in V\;|\; (c \cdot a, x_0 ) \in Y,(c \cdot a \cdot x_0,b_1) \in Y, (c \cdot a \cdot x_0 \cdot b_1, b_2) \in Y, \\
	&\quad (c \cdot a \cdot x_0 \cdot b, y_0) \in Y, (c^{-1}, c \cdot a \cdot x_0 \cdot b \cdot y_0) \in Y\}\\
	&=U_3(x_0,a) \cap U_3(b_1,a \cdot x_0) \cap U_3(b_2,a \cdot x_0 \cdot b_1) \cap U_3(y_0,a \cdot x_0 \cdot b) \\
	& \quad \cap  U_2(a \cdot x_0 \cdot b \cdot y_0).
\end{align*}
The definable set $Z'_0$ is nonempty for the same reason as the proof in Claim 1.
We take a point $c \in Z'_0$ and set
\begin{align*}
	U&:=\{(x,y) \in V \times V\;|\; (c \cdot a, x) \in Y,(c \cdot a \cdot x,b_1) \in Y, (c \cdot a \cdot x \cdot b_1, b_2) \in Y, \\
	&\quad  (c \cdot a \cdot x \cdot b, y) \in Y,  (c^{-1}, c \cdot a \cdot x \cdot b \cdot y) \in Y\}.
\end{align*}
The definable set $U$ is open and containing the point $(x_0,y_0)$ and the restriction of the map $h_{a,b}$ to $U$ is of class $\mathcal C^r$.
We can prove these facts in the same manner as Claim 1.
We omit the details.
\end{proof}

We consider the family of definable homeomorphisms $(\varphi_i:U_i:=g_iV \ni g_ig \mapsto g \in V)_{1 \leq i \leq m}$.
We want to show that it is a definable $\mathcal C^r$ structure on $G$.
We first prove that $\varphi_{ij}:\varphi_i(U_i \cap U_j) \to \varphi_j(U_i \cap U_j)$  given by $\varphi_{ij}(x)=\varphi_j(\varphi_i^{-1}(x))=g_j^{-1} \cdot g_i \cdot x$ is a definable $\mathcal C^r$ diffeomorphism whenever $1 \leq i,j \leq m$ and $U_i \cap U_j \neq \emptyset$.
Since $\varphi_{ji}$ is the inverse map of $\varphi_{ij}$, we have only to prove that $\varphi_{ij}$ is of class $\mathcal C^r$.
Note that $\varphi_i(U_i \cap U_j)=g_i^{-1} \cdot (g_i V \cap g_j V)=V \cap g_i^{-1} \cdot g_j V=\{x \in V\;|\; g_j^{-1} \cdot g_i \cdot x \in V\}=Z(g_j^{-1} \cdot g_i,e)$, where $e$ is the identity element in $G$. 
The map $\varphi_{ij}$ is of class $\mathcal C^r$ by Claim 1.
We have proven that $(G,(\varphi_{i})_{1 \leq i \leq m})$ is an abstract definable $\mathcal C^r$ manifold.

Let $\iota:G \to G$ be the inversion of $G$.
We next prove that $\iota$ is of class $\mathcal C^r$ as a map between abstract definable $\mathcal C^r$ manifolds.
For that purpose, we have to prove that the map $\iota'_{ij}=\varphi_j \circ \iota \circ \varphi_i^{-1}: \varphi_i(U_i \cap \iota^{-1}(U_j)) \to \varphi_j(\iota(U_i) \cap U_j)$ is of class $\mathcal C^r$ whenever $U_i \cap \iota^{-1}(U_j) \neq \emptyset$.
Note that $\iota'_{ij}(x)=g_j^{-1} \cdot (g_i \cdot x)^{-1}=g_j^{-1} \cdot x^{-1} \cdot g_i^{-1}$.
We have $\varphi_i(U_i \cap \iota^{-1}(U_j)) =g_i^{-1} \cdot (g_iV \cap \iota^{-1}(g_jV)) = g_i^{-1} \cdot (g_iV \cap V^{-1} g_j^{-1}) =V \cap (g_i^{-1}V g_j^{-1})$ because $V=V^{-1}$ by condition (b).
The map $\iota'_{ij}$ is the composition of the map $g_{g_j^{-1}, g_i^{-1}}$ defined in Claim 1, which is of class $\mathcal C^r$ by Claim 1,  with the restriction of $\iota$ to $V \cap (g_i^{-1}V g_j^{-1})$, which is also of class $\mathcal C^r$ by condition (b).
In conclusion, the map $\iota'_{ij}$ is of class $\mathcal C^r$.

The final task is to show that the multiplication is of class $\mathcal C^r $ as a map between abstract definable $\mathcal C^r$ manifolds.
We have only to show that the map $f'_{ijk}:=\varphi_k \circ f \circ (\varphi_i \times \varphi_j)^{-1}:(\varphi_i \times \varphi_j)((U_i \times U_j) \cap f^{-1}(U_k)) \to V_k$ whenever $(U_i \times U_j) \cap f^{-1}(U_k) \neq \emptyset$.
We have $f'_{ijk}(x,y)=g_k^{-1} \cdot g_i \cdot x \cdot g_j \cdot y$ and $(\varphi_i \times \varphi_j)((U_i \times U_j) \cap f^{-1}(U_k))=\{(x,y) \in V \times V\;|\; g_k^{-1} \cdot g_i \cdot x \cdot g_j \cdot y \in V\}=Z'(g_k^{-1} \cdot g_i,g_j)$.
Therefore,  $f'_{ijk}$ coincides with $h_{g_k^{-1} \cdot g_i,g_j}$ in the notation of Claim 2.
Claim 2 implies that $f'_{ijk}$ is of class $\mathcal C^r$.
\end{proof}

\end{document}